\documentclass[12pt]{amsart}
\usepackage{color}

\usepackage{enumerate}
\usepackage{graphics}

\usepackage{mathrsfs}
\usepackage{amssymb}

\usepackage[all]{xy}
\usepackage{graphicx}
\usepackage{tikz}  

\usepackage{amsmath}
\usepackage{amssymb}
\usepackage{amsfonts}
\usepackage{mathrsfs}
\usepackage{mathtools}

\newcommand{\R}{\mathbb{R}}
\newcommand{\T}{\mathbb{T}}
\newcommand{\C}{\mathbb{C}}

\newcommand{\A}{\mathcal{A}}
\newcommand{\B}{\mathcal{B}}

\usepackage{dsfont}

\newcommand{\norm}[1]{\Vert#1\Vert}
\newcommand{\bignorm}[1]{\bigl\Vert#1\bigr\Vert}
\newcommand{\Bignorm}[1]{\Bigl\Vert#1\Bigr\Vert}

\newcommand{\norme}[1]{\left\Vert #1\right\Vert}
\newcommand{\normeinf}[1]{\norme{#1}_{\infty}}

\newcommand{\MH}{\mathcal{M}(H^1(\R))}

\theoremstyle{plain}
\newtheorem{thm}{Theorem}[section]
\newtheorem{prop}[thm]{Proposition}

\newtheorem{lem}[thm]{Lemma}

\newtheorem{cor}[thm]{Corollary}

\theoremstyle{definition}
\newtheorem{df}[thm]{Definition}
\theoremstyle{remark}

\newtheorem{rq1}[thm]{Remark}

\numberwithin{equation}{section}

\begin{document}

\title[Fourier Multipliers on $H^1(\R)$ and functional calculus]
{$S^1$-bounded Fourier multipliers on $H^1(\R)$ and functional calculus for semigroups}

\author[L. Arnold]{Loris Arnold}
\email{larnold@impan.pl}
\address{Institute of Mathematics, Polish Academy of Sciences, Śniadeckich 8, 00-956 Warszawa, Poland}
\author[C. Le Merdy]{Christian Le Merdy}
\email{clemerdy@univ-fcomte.fr}
\address{Laboratoire de Math\'ematiques de Besan\c con, UMR 6623,
CNRS, Universit\'e Bourgogne Franche-Comt\'e,
25030 Besan\c{c}on Cedex, France}
\author[S. Zadeh]{Safoura Zadeh}
\email{jsafoora@gmail.com}
\address{School of Mathematics, University of Bristol,
Bristol BS8 1UG, United Kingdom}
\address{Institut des Hautes \'Etudes Scientifiques,
Bures-sur-Yvette 91440, France}

\date{\today}

\maketitle

\begin{abstract} 
Let $T\colon H^1(\R)\to H^1(\R)$ be a bounded Fourier multiplier on the
analytic Hardy space $H^1(\R)\subset L^1(\R)$ and let $m\in L^\infty(\R_+)$
be its symbol, that is, $\widehat{T(h)}=m\widehat{h}$ for all $h\in H^1(\R)$.
Let $S^1$ be the Banach space of all trace class operators on $\ell^2$. We show that
$T$ admits a bounded tensor extension 
$T\overline{\otimes}  I_{S_1}\colon H^1(\R;S^1)
\to  H^1(\R;S^1)$ if and only if
there exist a Hilbert space $\mathcal H$ and two functions
$\alpha, \beta \in L^\infty(\R_+;{\mathcal H})$
such that 
$m(s+t) = \langle\alpha(t),\beta(s)\rangle_{\mathcal H}$
for almost every $(s,t)\in\R_+^2$. Such Fourier multipliers are
called $S^1$-bounded and we let ${\mathcal M}_{S^1}(H^1(\R))$
denote the Banach space of all $S^1$-bounded Fourier multipliers.
Next we apply this result to functional calculus estimates, in two steps. 
First we introduce a new Banach 
algebra $\A_{0,S^1}(\C_+)$ of bounded analytic functions on
$\C_+ =\bigl\{z\in\C\, :\, {\rm Re}(z)>0\bigr\}$ and show that its dual 
space coincides with  ${\mathcal M}_{S^1}(H^1(\R))$. Second, given any 
bounded $C_0$-semigroup $(T_t)_{t\geq 0}$ on  Hilbert space, and 
any $b\in L^1(\R_+)$, we establish an 
estimate $\bignorm{\int_0^\infty b(t) T_t\, dt}\lesssim
\norm{L_b}_{\A_{0,S^1}}$, where $L_b$ denotes  
the Laplace transform of $b$. This improves 
previous functional calculus estimates recently obtained by 
the first two authors.
\end{abstract}

\vskip 0.8cm
\noindent
{\it 2000 Mathematics Subject Classification:} 42A45, 30H05 47D06, 47A60, 46L07

\smallskip
\noindent
{\it Key words:} Multipliers, Hardy spaces, Factorization,
Functional calculus, Semigroups.

\vskip 1cm

\section{Introduction}\label{Intro} 

\subsection{Presentation}
Let $H^1(\R)\subset L^1(\R)$ be the Hardy space consisting of all 
functions $h\in L^1(\R)$ whose Fourier transfom
$\widehat{h}$ vanishes on $\R_-$. A bounded Fourier multiplier
on $H^1(\R)$ is a bounded operator $T\colon H^1(\R)\to H^1(\R)$
such that there exists $m\in L^\infty(\R_+)$, called the symbol
of $T$, satisfying 
$\widehat{T(h)}=m\widehat{h}$ for all $h\in H^1(\R)$. Let 
$S^1$ be the Banach space of all trace class operators on $\ell^2$
and consider the $S^1$-valued Hardy space $H^1(\R;S^1)$.
We say that $T\colon H^1(\R)\to H^1(\R)$ is $S^1$-bounded
if $T\otimes I_{S^1}$ extends to a bounded map
$$
T\overline{\otimes}  I_{S_1}\colon H^1(\R;S^1)\longrightarrow H^1(\R;S^1).
$$
This is equivalent to $T$
being completely bounded in the framework of operator space theory
(see Remark \ref{CBH1} for details on this). There is no known
characterization of bounded Fourier multipliers
on $H^1(\R)$. The first goal of this paper is to prove that however, there is
a characterization of $S^1$-bounded (equivalently, completely bounded)
ones, as follows: a function $m\in L^\infty(\R_+)$ is the symbol
of an $S^1$-bounded Fourier multiplier
$T\colon H^1(\R)\to H^1(\R)$, with $\norm{T\overline{\otimes}  I_{S_1}}\leq C$,
if and only if there exist
a Hilbert space $\mathcal{H}$ and two 
functions $\alpha,\beta\in L^\infty(\R_+;\mathcal{H})$
such that $\norm{\alpha}_\infty\norm{\beta}_\infty\leq C$ and
$$
m(s+t)=\langle\alpha(t),\beta(s)\rangle_{\mathcal H}, \qquad\hbox{for\ 
a.e.-}(s,t)\in\R_+^2.
$$
This result, established in Theorem \ref{Main} below, has significant relationships 
with other Hilbert space factorization results.
First, this is a real line analogue of Pisier's theorem \cite[Theorem 6.2]{Pis1}
characterizing $S^1$-bounded Fourier multipliers on the discrete Hardy space
$H^1(\T)$ over the unit circle $\T$. Second, 
Theorem \ref{Main} has a strong analogy with
a famous result of Bozejko-Fendler \cite{BF} and Jolissaint \cite{J}
which asserts that for any locally compact group $G$,
a function $m\colon G\to\C$ induces a completely bounded 
Fourier multiplier on the group von Neumann algebra
of $G$, with completely bounded norm $\leq C$, if and only if there exist
a Hilbert space $\mathcal{H}$ and two functions $\alpha,\beta\in L^\infty(G;\mathcal{H})$
such that $\norm{\alpha}_\infty\norm{\beta}_\infty\leq C$ and
$m(st)=\langle\alpha(t),\beta(s)\rangle_{\mathcal H}$
for almost every $(s,t)\in G\times G$. Third, combining
\cite[Subsection 3.2]{Spronk} and Theorem \ref{Main}, we see
that $m\in L^\infty(\R_+)$ is the symbol
of an $S^1$-bounded Fourier multiplier on $H^1(\R)$ 
if and only if the function $(s,t)\mapsto m(s+t)$ 
induces a bounded measurable Schur multiplier on 
$B(L^2(\R_+))$. We refer to \cite{Coine, Daws} for more 
results relating multipliers and factorization theory.

The second goal of this paper is to prove new estimates 
on the functional calculus of generators of bounded 
$C_0$-semigroups on Hilbert space. The results we
obtain exploit Theorem \ref{Main} discussed above.
Consider the open half-plane
\begin{equation}\label{C+}
\C_+ =\bigl\{z\in\C\, :\, {\rm Re}(z)>0\bigr\}.
\end{equation}
Let $H$ be a Hilbert space and let $-A$ be the  
infinitesimal generator of a bounded $C_0$-semigroup $(T_t)_{t\geq 0}$
on $H$. 
To any $b\in L^1(\R_+)$, one may associate the operator 
$\Gamma(A,b)\in B(H)$ defined by 
$$
[\Gamma(A,b)](x) = \int_0^\infty b(t) T_t(x)\, dt,\qquad x\in H.
$$
Formally this operator can be regarded as $L_b(A)$, where
$$
L_b\colon \C_+\to \C,\quad
L_b(z)=\int_0^\infty b(t)e^{-tz}\, dt,
$$
is the Laplace transform of $b$.
The functional calculus issue is to estimate the norm
of $\Gamma(A,b)$ in terms of a ``sharp and useful" norm
of $L_b$. The theory of $H^\infty$-functional calculus 
provides specific, interesting results in the case when $(T_t)_{t\geq 0}$
is a bounded analytic semigroup \cite{MI, LM0}. 
In the general setting discussed here, Haase \cite[Corollary 5.5]{Haase2}
proved a remarkable
estimate $\norm{\Gamma(A,b)}\lesssim\norm{L_b}_{\B_0}$ in terms 
of a suitable Besov algebra $\B_0(\C_+)$. We refer to \cite{BGT1, BGT2}
for more on this estimate, a concrete description
of the Besov functional calculus, and applications.
In the recent paper \cite{AL}, the first two authors
introduced a new Banach algebra $\A_0(\C_+)$ of
analytic functions on $\C_+$ and proved that
$\norm{\Gamma(A,b)}\leq\bigl(\sup_{t\geq 0}\norm{T_t}\bigr)^2\,
\norm{L_b}_{\A_0(\C_+)}$ for all $b\in L^1(\R_+)$.
This improved Haase's estimate quoted above.

In the present paper, we define a Banach algebra $\A_{0,S^1}(\C_+)$
such that $\A_{0}(\C_+)\subset \A_{0,S^1}(\C_+)$ contractively and
\begin{equation}\label{NewEst}
\norm{\Gamma(A,b)}\leq\Bigl(\sup_{t\geq 0}\norm{T_t}\Bigr)^2\,
\norm{L_b}_{\A_{0,S^1}(\C_+)},\qquad b\in L^1(\R_+).
\end{equation}
Moreover $\A_{0}(\C_+)\not=\A_{0,S^1}(\C_+)$, so that 
(\ref{NewEst}) is a genuine improvement of \cite[Theorem 4.4]{AL}.

We now briefly describe the contents of the next sections. 
Section 2 contains preliminary results on vector valued Hardy spaces
and Hilbert space factorizations. Section 3 is devoted to the
characterization of $S^1$-bounded Fourier multipliers described in 
the first paragraph of this subsection. In Section 4
we introduce and study the Banach algebra $\A_{0,S^1}$ 
and devise an isometric 
identification
$$
\A_{0,S^1}^*\,\simeq\, {\mathcal M}_{S^1}(H^1(\R)),
$$
where ${\mathcal M}_{S^1}(H^1(\R))$ denotes the space of
$S^1$-bounded Fourier multipliers on $H^1(\R)$ equipped with 
its natural norm. Functional calculus results, including (\ref{NewEst}),  are established in 
Section 5.

\subsection{Notations, conventions and definitions.} 
All Banach spaces considered in this paper are complex ones.

Recall the definition of $\C_+$ from (\ref{C+}). We will also use the half-plane
$$
P_+ =\bigl\{z\in\C\, :\, {\rm Im}(z)>0\bigr\}.
$$
For any Banach space $Z$, we let ${\rm Hol}(P_+;Z)$ denote the vector space
of all holomorhic functions from $P_+$ into $Z$.

Let $X$ be a Banach space.
The duality action of $x^*\in X^*$ on $x\in X$ is denoted by
$\langle x^*,x \rangle_{X^*,X}$ or simply by
$\langle x^*,x \rangle$ if there is no risk of confusion. 

For any Hilbert space $\mathcal H$, we let 
$\langle\,\cdotp,\,\cdotp\rangle_{\mathcal H}$ denote
the scalar product on $\mathcal H$ that we assume, by convention,
linear in the first variable and anti-linear in the second
variable.

Let $(\Omega,\mu)$ be a $\sigma$-finite measure space, let $Z$ be a Banach space and let 
$1\leq p\leq\infty$. We let $L^p(\Omega;Z)$ denote the Bochner space 
of all measurable functions $h\colon \Omega\to Z$
(defined up to almost everywhere zero functions)  such that the
norm function $t\mapsto \norm{h(t)}_Z$ belongs to $L^p(\Omega)$ 
(see e.g. \cite[Chapters I and II]{DU}). 
The norm of an element $h\in L^p(\Omega;Z)$, which is equal
to the $L^p(\Omega)$-norm of  $t\mapsto \norm{h(t)}_Z$, will simply be denoted  
by $\norm{h}_{p}$ if the underlying space $Z$ is clear.
We will use the fact that when $p\not=\infty$, $L^p(\Omega)\otimes Z$
is dense in $L^p(\Omega;Z)$.

For any Hilbert space $\mathcal H$, we let $S^\infty(\mathcal H)$
denote the Banach space of all compact operators $\mathcal H\to \mathcal H$. Likewise,
we let $S^1(\mathcal H)$ denote the Banach space of all trace class operators $\mathcal H\to \mathcal H$.
We let $tr$ denote the usual trace on $S^1(\mathcal H)$ and we recall that 
the duality pairing 
$$
\langle b,a\rangle_{S^1,S^\infty} = tr(ab),\qquad
a\in S^\infty(\mathcal H), \ b\in S^1(\mathcal H),
$$
induces an isometric isomorphism $S^\infty(\mathcal H)^*\simeq S^1(\mathcal H)$, see e.g. \cite[Chapter 3]{Con}.
We simply let $S^\infty=S^\infty(\ell^2)$ and  $S^1=S^1(\ell^2)$. Accordingly,
we let $S^2$ denote the Hilbert space of Hilbert-Schmidt operators on $\ell^2$. Also for any integer $n\geq 1$,
we let $M_n$ denote the space of $n\times n$ matrices equipped with its usual  operator norm
(i.e. $M_n=S^\infty(\ell^2_n)$) and we set $S^1_n=S^1(\ell^2_n)$.

The Fourier transform of any $h\in L^1(\R)$ is defined
by
\[
\widehat{h}(u) = \int_{-\infty}^\infty h(t)e^{-itu}\,dt,
\qquad u\in\R.
\]
We keep the same notation $\widehat{h}$ to denote the Fourier transform of any
$h\in L^1(\R;Z)+L^\infty(\R;Z)$ whenever $Z$ is a Banach space. We recall that for any 
Hilbert space $\mathcal H$, the Fourier transform on $L^2(\R;\mathcal H)$ satisfies
\begin{equation}\label{Plancherel}
\norm{\widehat{h}}_{L^2(\R;\mathcal H)}\,=\,\sqrt{2\pi}\,
\norm{h}_{L^2(\R;\mathcal H)},\qquad h\in L^2(\R;\mathcal H).
\end{equation}
Sometimes it will be convenient to write $\mathcal{F}(h)$ instead of 
$\widehat{h}$. Wherever it makes sense, we will
use $\mathcal{F}^{-1}$ to denote the inverse Fourier transform.

We let $C_b(\R_+^*)$ denote the Banach space of all continuous
bounded functions from $\R_+^*$ into $\C$, equipped
with the supremum norm. We will regard it as a subspace of 
$L^\infty(\R_+)$, the embedding
$$
C_b(\R_+^*)\subset L^\infty(\R_+)
$$
being an isometry.

For any Banach spaces $Y,Z$, we let $B(Y,Z)$ denote the Banach space of
all bounded operators from $Y$ into $Z$. We write $B(Y)$ instead of $B(Y,Y)$
if $Z=Y$ and we let $I_Y \in B(Y)$ denote the identity operator on $Y$.

The following classical notion will play a key role in this paper.

\begin{df}\label{Quotient}
{\it Let $Y,Z$ be two Banach spaces. 
We say that a bounded operator $T\colon Y\to Z$ is a quotient map 
if $\norm{T}\leq 1$ and for any $z\in Z$ with $\norm{z}<1$, there
exist $y\in Y$ such that $T(y)=z$ and $\norm{y}<1$. }
\end{df}

Let  $T\colon Y\to Z$ and let 
$\widetilde{T}\colon \frac{Y}{{\rm Ker}(T)}\to Z$ be the operator
induced by $T$ after factorization through its kernel. 
It is clear from the definition that
$T$ is a quotient map if and only if 
$\widetilde{T}$ is an isometry. It is well-known 
that $T$ is a quotient map if and only if
$T^*\colon Z^*\to Y^*$ is an isometry. In this case,
${\rm Ran}(T^*)$ is $w^*$-closed and equal to ${\rm Ker}(T)^\perp$.

\section{Preliminary results}

\subsection{The dual space $\Gamma_2(L^1,L^\infty)$}\label{Gamma2} In this
subsection, $X,Y,Z$ will denote arbitrary Banach spaces. We let
$\norm{\,\cdotp}_\wedge$ denote the projective norm on
$X\otimes Y$ and we let $X\widehat{\otimes} Y$ denote 
the completion of $(X\otimes Y,\norm{\,\cdotp}_\wedge)$,
called the projective tensor product of 
$X$ and $Y$. For any bounded operator $S\colon Y\to X^*$, there 
exists a (necessarily unique) 
$\xi_S\in (X\widehat{\otimes} Y)^*$ such that 
$$
\xi_S(x\otimes y)= \langle S(y),x\rangle_{X^*,X},\qquad x\in X,\ y\in Y.
$$
Then $\norm{\xi_S}=\norm{S}$, and hence the mapping $S\mapsto \xi_S$ gives
rise to an isometric identification
\begin{equation}\label{Proj}
(X\widehat{\otimes} Y)^*\,\simeq\, B(Y,X^*).
\end{equation}
We refer to either \cite[Section VIII.2, Theorem 1 $\&$ Corollary 2]{DU} or
\cite[Theorem 2.9]{Ryan} for these classical facts.

We say that a bounded operator
$S\colon Y\to Z$ factors through Hilbert space 
if there exist a Hilbert space $\mathcal H$
and two bounded operators $A\colon Y\to \mathcal H$
and $B\colon \mathcal H\to Z$ such that $BA=S$. 
We denote by $\Gamma_2(Y,Z)$ the space of all such operators.
For any $S\in \Gamma_2(Y,Z)$, we set 
$$
\gamma_2(S)=\inf\bigl\{\norm{A}\norm{B}\bigr\},
$$
where the infimum runs over all factorizations
of $S$ as above. It turns out that 
$\gamma_2$ is a norm on $\Gamma_2(Y,Z)$ and that the latter
is a Banach space. We refer to \cite[Chapter 3]{Pis1} for
details and various characterizations of operators
which factor through Hilbert space.

Assume that $Z=X^*$ is a dual space. 
According to \cite[Theorem 5.3]{Pis1},
there exists a
norm $\gamma_2^* \leq\norm{\,\cdotp}_\wedge$  on
$X\otimes Y$ such that 
if we let $X\otimes_{\gamma_2^*} Y$ 
denote the completion of $(X\otimes Y,\gamma_2^*)$,
then (\ref{Proj}) induces an isometric identification
\begin{equation}\label{gamma2*}
(X\otimes_{\gamma_2^*} Y)^*\,\simeq\, \Gamma_2(Y,X^*).
\end{equation}
We will not use the definition of $\gamma_2^*$ (we refer
the interested reader to \cite[Chapter 5]{Pis1} for information).

We now focus on the case when $X,Y$ are $L^1$-spaces.
Let $(\Omega_1,\mu_1)$ and $(\Omega_2,\mu_2)$ be
two $\sigma$-finite measure spaces. 
For any
$\varphi\in L^\infty(\Omega_1\times\Omega_2)$, one
may define an operator 
\begin{equation}\label{Kernel1}
S_\varphi\colon L^1(\Omega_2)\longrightarrow
L^\infty(\Omega_1)
\end{equation}
by 
\begin{equation}\label{Kernel2}
S_\varphi(f) = \,\int_{\Omega_2} \varphi(\,\cdotp,t)f(t)\, d\mu_2(t),\qquad 
f\in L^1(\Omega_2).
\end{equation}
The function $\varphi$ is called the kernel of $S_\varphi$.
Applying (\ref{Proj}) with $X=L^1(\Omega_1)$ and 
$Y=L^1(\Omega_2)$, together with the standard identity
$L^1(\Omega_1)\widehat{\otimes}L^1(\Omega_2)\simeq 
L^1(\Omega_1\times\Omega_2)$, we obtain 
an isometric $w^*$-homeomorphic identification
\begin{equation}\label{L1Linf}
L^\infty(\Omega_1\times\Omega_2)\,\simeq\, B(L^1(\Omega_2),
L^\infty(\Omega_1)).
\end{equation}
It is easy to check that this identification 
is given by the mapping $\varphi\to S_\varphi$. In particular,
any element of $B(L^1(\Omega_2),
L^\infty(\Omega_1))$ is of the form $S_\varphi$ for some
(unique) kernel $\varphi$.

We  aim at describing $\Gamma_2(L^1(\Omega_2),
L^\infty(\Omega_1))$ in terms of kernels.
We let $\mathcal{V}_2(\Omega_1\times\Omega_2)$ be the set of 
all $\varphi\in  L^\infty(\Omega_1\times\Omega_2)$ that essentially factor through
Hilbert space, in the sense that
there exist a Hilbert space $\mathcal H$ and two functions
$\alpha\in L^\infty(\Omega_2;{\mathcal H})$
and $\beta\in L^\infty(\Omega_1;{\mathcal H})$ such that 
$$
\varphi(s,t) = \langle\alpha(t),\beta(s)\rangle_{\mathcal H}\qquad\hbox{for\ 
a.e.-}(s,t)\in\Omega_1\times\Omega_2.
$$
For any such $\varphi$, we set
$$
\nu_2(\varphi)\,=\,\inf\bigl\{\norm{\alpha}_{L^\infty(\Omega_2;
{\mathcal H})}\norm{\beta}_{L^\infty(\Omega_1;{\mathcal H})}\bigr\},
$$
where the infimum runs over all factorizations
of $\varphi$ as above. 
The next lemma shows that $\nu_2$ is a norm  on $\mathcal{V}_2
(\Omega_1\times\Omega_2)$ and that the latter
is a (dual) Banach space. 

\begin{lem}\label{V2}\  

\begin{itemize}
\item [(1)] Let  $\varphi\in  L^\infty(\Omega_1\times\Omega_2)$. Then $S_\varphi$
belongs to $\Gamma_2(L^1(\Omega_2),
L^\infty(\Omega_1))$ if and only if $\varphi$ belongs to $\mathcal{V}_2(\Omega_1\times\Omega_2)$.
We further  have $\nu_2(\varphi) = \gamma_2(S_\varphi)$ in this case.
\item [(2)] The unit ball
$$
\bigl\{\varphi\in \mathcal{V}_2(\Omega_1\times\Omega_2)\, :\, \nu_2(\varphi)\leq 1\bigr\}
$$
is $w^*$-closed in $L^\infty(\Omega_1\times\Omega_2)$.
\end{itemize}
\end{lem}

\begin{proof}
To prove (1), let  $\varphi\in  L^\infty(\Omega_1\times\Omega_2)$
and assume that $S_\varphi$ factors through Hilbert space. There exist 
a Hilbert space $\mathcal H$
and two bounded operators $A\colon L^1(\Omega_2)\to \mathcal H$
and $B\colon \mathcal H\to  L^\infty(\Omega_1)$ such that $BA=S_\varphi$. 
Let $B_*\colon L^1(\Omega_1)\to{\mathcal H}^*$ be the restriction
of the adjoint of $B$ to $L^1(\Omega_1)$. Then let 
$\overline{B_*}\colon L^1(\Omega_1)\to{\mathcal H}$ be defined
as follows: for any $f\in L^1(\Omega_1)$, 
$$
\langle\zeta,\overline{B_*}(f)\rangle_{\mathcal H}\,=\,\langle B_*(\overline{f}),\zeta\rangle_{{\mathcal H}^*,{\mathcal H}},
\qquad \zeta\in\mathcal H.
$$
This is a linear map and $\norm{\overline{B_*}}=\norm{B}$.

Since any Hilbert space has the Radon-Nikodym property \cite[Corollary IV.1.4]{DU}, 
there exists 
a function $\alpha\in L^\infty(\Omega_2;{\mathcal H})$ such that 
$$
A(f_2) = \int_{\Omega_2} \alpha(t)f_2(t)\, d\mu_2(t),\qquad f_2\in L^1(\Omega_2),
$$
by \cite[Theorem III.1.5]{DU}. Moreover $\norm{\alpha}_{L^\infty(\Omega_2;{\mathcal H})}=\norm{A}$.
Likewise, there exists  
a function $\beta\in L^\infty(\Omega_1;{\mathcal H})$ such that 
$$
\overline{B_*}(f_1) = \int_{\Omega_1} \beta(s)f_1(s)\, d\mu_1(s),
\qquad f_1\in L^1(\Omega_1),
$$
and $\norm{\beta}_{L^\infty(\Omega_1;{\mathcal H})}=\norm{B}$.

Now  for any  $f_1\in L^1(\Omega_1)$ and $f_2\in L^1(\Omega_2)$, we have
\begin{align*}
\int_{\Omega_1\times\Omega_2} \varphi(s,t) f_1(s) f_2(t)\,d\mu_1(s)\, d\mu_2(t)\,
& =\, \langle S_\varphi(f_2),f_1\rangle_{L^\infty(\Omega_1), L^1(\Omega_1)}\\
& =\, \langle A(f_2),\overline{B_*}(\overline{f_1})\rangle_{{\mathcal H}}\\
& =\,\biggl\langle \int_{\Omega_2} \alpha(t)f_2(t)\, d\mu_2(t),
\int_{\Omega_1} \beta(s) \overline{f_1(s)}\, d\mu_1(s)\biggr\rangle_{\mathcal H}\\
& =\, \int_{\Omega_1\times\Omega_2} \langle \alpha(t),\beta(s)\rangle_{\mathcal H} f_2(t) f_1(s)\,d\mu_1(s)\, d\mu_2(t).
\end{align*}
This implies that $\varphi(s,t)= \langle \alpha(t),\beta(s)\rangle_{\mathcal H}$ for a.e.-$(s,t)$ in $\Omega_1\times\Omega_2$.
This shows that  $\varphi$ belongs to $\mathcal{V}_2(\Omega_1\times\Omega_2)$ and
that $\nu_2(\varphi)\leq \norm{\alpha}_\infty\norm{\beta}_\infty=\norm{A}\norm{B}$. Passing to the infimum
over all possible factorizations of $S_\varphi$, we obtain that $\nu_2(\varphi)\leq\gamma_2(S_\varphi)$.

Reversing the arguments, 
we obtain that if $\varphi\in\mathcal{V}_2(\Omega_1\times\Omega_2)$,
then $S_\varphi$ belongs to the space $\Gamma_2(L^1(\Omega_2),
L^\infty(\Omega_1))$, with  $\gamma_2(S_\varphi)\leq \nu_2(\varphi).$

To prove (2), note that using the 
standard identity
$L^1(\Omega_1)\widehat{\otimes}L^1(\Omega_2)\simeq 
L^1(\Omega_1\times\Omega_2)$, we have a natural contractive embedding
$$
\kappa\colon L^1(\Omega_1\times\Omega_2)\longrightarrow L^1(\Omega_1)\otimes_{\gamma_2^*} L^1(\Omega_2).
$$
Using (\ref{gamma2*}), we can see its adjoint map as $\kappa^*\colon
\Gamma_2(L^1(\Omega_2),
L^\infty(\Omega_1))\to L^\infty(\Omega_1\times\Omega_2)$. It follows from 
(1) that the unit ball of $\mathcal{V}_2(\Omega_1\times\Omega_2)$ is equal to the 
image of the unit ball of $\Gamma_2(L^1(\Omega_2),
L^\infty(\Omega_1))$ under $\kappa^*$. Thus it is a $w^*$-closed subset of 
 $L^\infty(\Omega_1\times\Omega_2)$.
\end{proof}

\subsection{Vector valued Hardy spaces and duality}\label{Hardy}

We recall that $H^1(\R)\subset L^1(\R)$ is the subspace of all $h\in L^1(\R)$ such that 
$\widehat{h}(u)=0$ for all $u\leq 0$. 
More generally for any $1\leq p\leq\infty$, the Hardy 
space $H^p(\R)$ is, by definition, the subspace
of all functions in $L^p(\R)$ whose Fourier transform has support in $\R_+$. Equivalently,
for any $1\leq p<\infty$,
$H^p(\R)$ is the closure of $H^1(\R)\cap L^p(\R)$ in $L^p(\R)$ whereas
$H^\infty(\R)$ is the $w^*$-closure of $H^1(\R)\cap L^\infty(\R)$ in $L^\infty(\R)$.
We refer e.g. to \cite[Chapter II]{Gar} or \cite[Chapter 8]{Hof} for basics on $H^p(\R)$.

We will need vector valued Hardy spaces on the real line and a few preliminary results 
that we establish in this section. Let $Z$ be a Banach space.
For any $1\leq p<\infty$, we
let $H^p(\R;Z)\subset L^p(\R;Z)$ denote the closure of 
$H^p(\R)\otimes Z$ in $L^p(\R;Z)$.

Given a harmonic function $\phi\colon P_+\to Z$
and  $v>0$,  let $\phi^v\colon\R\to Z$ be defined by $\phi^v(u)=\phi(u+iv)$, $u\in\R$.
Then we let ${\bold h}^1(P_+;Z)$ be the space of all such functions
such that $\phi^v\in L^1(\R;Z)$ for all $v>0$ and 
$(\phi^v)_{v>0}$ is bounded in $L^1(\R;Z)$.
It turns out that  ${\bold h}^1(P_+;Z)$ is a Banach space for the norm
$$
\norme{\phi}_{{\bold h}^1(P_+;Z)} = \sup_{v>0}\bignorm{\phi^v}_{L^1(\R;Z)}
=\sup_{v>0}\Bigl(\int_{-\infty}^{\infty}\norme{\phi(u+iv)}_Z\, du\Bigr).
$$
We let ${\bold H}^1(P_+;Z)={\bold h}^1(P_+;Z)\cap {\rm Hol}(P_+;Z)$ be the
subspace of ${\bold h}^1(P_+;Z)$ formed of all holomorphic functions, equiped with the
${\bold h}^1(P_+;Z)$-norm. Then ${\bold H}^1(P_+;Z)$ is closed (hence a 
Banach space) and by construction,
\begin{equation}\label{H1h1}
{\bold H}^1(P_+;Z)\subset {\bold h}^1(P_+;Z).
\end{equation}
We refer  the reader to \cite[Section 4.6]{Pis2} for details and more information 
on these spaces (as well as on 
$p$-analogs that we will not use in this paper).

Let $M(\R;Z)$ be the Banach space of all bounded $Z$-valued Borel measures on $\R$,
as defined in \cite[Section 2.1]{Pis2}.
For any $\mu\in M(\R;Z)$, let $\phi_\mu\colon P_+\to Z$ be its Poisson
integral defined by
$$
\phi_\mu(u+iv) \,=\,\frac{1}{\pi}\,
\int_{-\infty}^{\infty}\frac{v}{(u-t)^2+v^2}\, d\mu(t), \qquad (u,v)\in\R\times\R_+^*.
$$
Then $\phi_\mu$ belongs to ${\bold h}^1(P_+;Z)$ and 
$\norme{\phi_\mu}_{{\bold h}^1(P_+;Z)} =\norm{\mu}_{M(\R;Z)}$. Thus $\mu\mapsto \phi_\mu$
induces an isometric embedding of $M(\R;Z)$
into ${\bold h}^1(P_+;Z)$. Remarkably,  this embedding is 
an equality, which yields an isometric identification 
\begin{equation}\label{h1}
M(\R;Z)\,\simeq\, {\bold h}^1(P_+;Z).
\end{equation}
We refer to \cite[p. 145 and Theorem 3.2]{Pis2} for the proofs.
Combining (\ref{H1h1}) and  (\ref{h1}), we obtain an isometric embedding
$$
{\bold H}^1(P_+;Z)\subset M(\R;Z).
$$

We regard $L^1(\R;Z)$ as a subspace of $M(\R;Z)$ in a standard way by identifying
any $h\in L^1(\R;Z)$ with the measure $\mu\in M(\R;Z)$ defined by
$$
\mu(E) = \int_E h(t)\, dt
$$
for any Borel set $E\subset\R$. It is well-known (and easy to check)
that the Poisson integral of
any element of $H^1(\R;Z)$ is analytic, hence belongs to ${\bold H}^1(P_+;Z)$.
Thus the isometric isomorphism (\ref{h1}) gives rise to an isometric embedding
\begin{equation}\label{InclusionHardy}
H^1(\R;Z)\subset {\bold H}^1(P_+;Z).
\end{equation}
In the case when $Z=\C$, this embedding is an equality and the resulting identification
\begin{equation}\label{H1H1}
H^1(\R)\simeq {\bold H}^1(P_+)
\end{equation}
is a cornerstone of the theory of classical (=scalar
valued) Hardy spaces, see \cite{Gar, Hof}. In the vector valued case, the 
inclusion (\ref{InclusionHardy}) may be strict. We will come back to this
issue in Remark \ref{ARNP} below.

In the sequel we simply 
write $M(\R)=M(\R;\C)$ for the space of 
bounded Borel measures on $\R$. 
Let $C_0(\R)$ be the Banach algebra  of 
continuous functions $\R\to\C$ which vanish at infinity.
We will use the identification $M(\R)\simeq C_0(\R)^*$ 
(Riesz's Theorem) provided by the duality
pairing
\begin{equation}\label{Riesz}
\langle\mu,f\rangle \,=\, \int_{\tiny{\R}} f(-t)\, d\mu(t),\qquad
\mu\in M(\R),\ f\in C_0(\R).
\end{equation}
In the sequel we let $H^{\infty}_{-}(\R)=\{f(-\,\cdotp)\, :\, f\in H^{\infty}(\R)\}$ 
and let
$Z_0=C_0(\R)\cap H^{\infty}_{-}(\R)$. It was noticed in \cite[Section 3.3]{AL} that
\begin{equation}\label{Z0}
Z_0^\perp =H^1(\R).
\end{equation}
(The minus sign in (\ref{Riesz}) is there to allow this simple description of
$H^1(\R)$ as a dual space.) This
yields an isometric identification
$$
\biggr(\frac{C_0(\R)}{Z_0}\biggl)^*\,\simeq\, H^1(\R).
$$

Our aim is now to establish a vector valued analog of this result.

Let $X$ be a Banach space. We let 
$C_0(\R;X)$ be the Banach space of 
continuous functions from $\R$ into $X$ which vanish at infinity, 
equipped with the supremum norm. Note that $C_0(\R)\otimes X$
is dense in $C_0(\R;X)$ and that the latter identifies with 
$C_0(\R) \overset{\vee}{\otimes} X$, where  $\overset{\vee}{\otimes}$ denotes
the injective tensor product (see e.g. \cite[Chapter VIII]{DU}).
Since any element
of  $M(\R;X^*)$ is a regular measure, it follows from 
\cite[Theorem 1]{Mez} that Riesz's Theorem extends to 
an isometric identification
\begin{equation}\label{Meziani}
C_0(\R;X)^*\,\simeq\,M(\R;X^*).
\end{equation}
More precisely, for any $\mu\in M(\R;X^*)$ and $x\in X$,
let $\mu^x\in M(\R)$ be defined by setting 
$\mu^x(E)=\langle \mu(E),x\rangle$ for any Borel set $E\subset\R$.
We may define
\begin{equation}\label{Meziani2}
\langle \mu, f\otimes x\rangle\,=\, \int_{\tiny{\R}} f(-t)\, d\mu^x(t)\,.
\end{equation}
By linearity, this extends to a duality pairing between
$M(\R;X^*)$ and $C_0(\R)\otimes X$.
Then by (\ref{Meziani}), we mean that this duality pairing 
yields an isometric isomorphism between $M(\R;X^*)$ and $C_0(\R;X)^*$.

\begin{prop}\label{H1Dual2} 
Let $X$ be a Banach space. Through the identification (\ref{Meziani}),
we have 
\begin{equation}\label{H1Dual3} 
\bigl(Z_0\otimes X\bigr)^\perp = {\bold H}^1(P_+;X^*),
\end{equation}
and hence an isometric identification
\begin{equation}\label{H1Dual4} 
\biggr(\frac{C_0(\R;X)}{Z_0 \overset{\vee}{\otimes} X}\biggl)^*
\,\simeq\, {\bold H}^1(P_+;X^*).
\end{equation}
\end{prop}

\begin{proof} Since $Z_0 \overset{\vee}{\otimes} X$
is the closure of $Z_0\otimes X$ in $C_0(\R;X)$,
the assertion (\ref{H1Dual4}) readily follows from (\ref{H1Dual3}). 
Thus we only need to prove (\ref{H1Dual3}). 

Let $\mu\in M(\R;X^*)$ and let $\phi_\mu$ be its Poisson integral. It is clear that 
for any $x\in X$, the function $\langle \phi_\mu(\,\cdotp),x\rangle\colon
P_+\to\C$ taking any $z\in P_+$ to $\langle \phi_\mu(z),x\rangle$
is the Poisson integral of $\mu^x$. 

By (\ref{Meziani2}), 
$\mu\in \bigl(Z_0\otimes X\bigr)^\perp$ if and only if 
$\mu^x\in Z_0^\perp$ for any $x\in X$. Hence by (\ref{Z0}),
$\mu\in \bigl(Z_0\otimes X\bigr)^\perp$ if and only if 
$\mu^x\in H^1(\R)$ for any $x\in X$. Applying (\ref{H1H1}), we obtain that 
$$
\mu\in \bigl(Z_0\otimes X\bigr)^\perp\,\Longleftrightarrow\,
\forall\, x\in X,\quad
\langle \phi_\mu(\,\cdotp),x\rangle\,\in  {\bold H}^1(P_+).
$$
Since $\phi_\mu\in {\bold h}^1(P_+;X^*)$, the right-hand side 
property is equivalent to the fact that $\phi\in {\rm Hol}(P_+;X^*)$.
Hence it is equivalent to $\phi\in {\bold H}^1(P_+;X^*)$.
\end{proof}

\begin{rq1}\label{ARNP}
We say that $X^*$ has the analytic Radon-Nikodym property
(ARNP in short) if  the 
inclusion (\ref{InclusionHardy}) is an equality. We refer to \cite[Chapter 4]{Pis2} for information 
on this property and equivalent formulations. All Banach spaces with 
the so-called  Radon-Nikodym property (RNP) have ARNP. We recall that 
reflexive spaces and separable duals have RNP \cite[Section VII.7]{DU}, hence ARNP.
Further preduals of von Neumann algebras (in particular, $L^1$-spaces) 
have ARNP, see \cite{BD, HP}.
\end{rq1}

Since $S^1\simeq (S^\infty)^*$ is a separable dual,
it follows from Proposition \ref{H1Dual2} and Remark \ref{ARNP} that
\begin{equation}\label{H1S1}
\biggr(\frac{C_0(\R;S^\infty)}{Z_0 \overset{\vee}{\otimes} S^\infty}\biggl)^*
\,\simeq\, H^1(\R;S^1).
\end{equation}

\subsection{Sarason's Theorem on the real line}\label{Sarason}

Let $\T$ denote the unit circle of the complex plane, 
equipped with its normalized Haar measure $d\nu$.
For any $1\leq p<\infty$, let $H^p(\T)\subset L^p(\T)$ denote the
Hardy space of functions whose Fourier coefficients vanish on negative integers.
For any Banach space $Z$, we let  $H^p(\T;Z)\subset L^p(\T;Z)$ denote the closure of 
$H^p(\T)\otimes Z$ in $L^p(\T;Z)$. These definitions mimic those of $H^p(\R;Z)$
given above.

For any $g_1,g_2\in H^2(\T;S^2)$, the function $g_1g_2$ taking
any $\theta \in\T$ to $g_1(\theta)g_2(\theta)$ belongs to
$H^1(\T;S^1)$ and $\norm{g_1g_2}_{H^1(\T;S^1)}\leq
\norm{g_1}_{H^2(\T;S^2)}\norm{g_2}_{H^2(\T;S^2)}$.
A remarkable theorem of Sarason \cite{Sar} (see also \cite{HP}) asserts
that conversely, for any $g\in H^1(\T;S^1)$, there exist 
 $g_1,g_2\in H^2(\T;S^2)$ such that $g=g_1g_2$ and
$\norm{g_1}_{H^2(\T;S^2)}\norm{g_2}_{H^2(\T;S^2)}=
\norm{g}_{H^1(\T;S^1)}$.

The following real line version of Sarason's Theorem is well-known to
specialists. We include a proof for the sake of completeness.

\begin{prop}\label{Sar}
For any $h$ in $H^1(\mathbb{R};S^1)$, there exist $h_1$ and 
$h_2$ in $H^2(\mathbb{R};S^2)$ such that $h=h_1h_2$ and 
$\norm{h_1}_{H^2(\R;S^2)}\norm{h_2}_{H^2(\R;S^2)}=
\norm{h}_{H^1(\R;S^1)}$.
\end{prop}

\begin{proof}
It is observed in \cite[pp. 121-122]{Hof} that for any $g\in L^1(\mathbb{T})$, the
function $t\mapsto \frac{1}{1+t^2}g\bigl(\frac{t-i}{t+i}\bigr)$ belongs to 
$L^1(\mathbb{R})$ and that we have
\begin{equation}\label{ChangeVar}
\frac{1}{\pi}\int_{-\infty}^\infty g\Bigl(
\frac{t-i}{t+i}\Bigr)\,\frac{dt}{1+t^2}\,=\,\int_{\T} g(\theta)\, d\nu(\theta).
\end{equation}
Further it follows from \cite[Chapter 8]{Hof} (see also \cite[II, Proposition 5.1]{Ric})
that for any $1\leq p<\infty$, we
have an onto isometry $G_p\colon H^p(\mathbb{T})\to H^p(\mathbb{R})$ given by 
$$
[G_p(f)](t)\,=\,\frac{1}{\pi^{\frac{1}{p}}(t+i)^{\frac{2}{p}}}f\Bigl(\frac{t-i}{t+i}\Bigr),
\qquad f\in H^p(\mathbb{T}),\, t\in\R.
$$

Let $Z$ be an arbitrary Banach space. 
Let $f\in H^p(\mathbb{T})\otimes Z$. Using \eqref{ChangeVar}
with $g(\theta)= \|f(\theta)\|^p$, we have 
$$
\frac{1}{\pi}\int_{-\infty}^\infty
\Bignorm{f\Bigl(\frac{t-i}{t+i}\Bigr)}^p\,\frac{dt}{1+t^2}
\,=\,\int_{\T}\|f(\theta)\|^p \,d\nu(\theta).
$$
Since $\bigl\vert (t+i)^{\frac{2}{p}}\bigr\vert^p = 1+t^2$,
this implies that $G_p\otimes I_Z$ extends to an isometry, denoted by
$G_p\overline{\otimes} I_Z$, from $H^p(\mathbb{T}; Z)$ into $H^p(\mathbb{R}; Z)$. 
Since $G_p\overline{\otimes} I_Z$ is an isometry, its range is closed. Since
$G_p$ is onto,
the range of $G_p\overline{\otimes} I_Z$ contains $H^p(\mathbb{R})\otimes Z$. 
Therefore, the range of $G_p\overline{\otimes} I_Z$ is equal
to $H^p(\mathbb{R};Z)$, that is, $G_p\overline{\otimes} I_Z$ is onto.

We now prove the announced factorization result, using $G_1$ and $G_2$.
Let $h\in H^1(\mathbb{R};S^1)$. We may define 
$g=(G_1\overline{\otimes}I_{S^1})^{-1}(h)\in H^1(\mathbb{T};S^1).$ 
Using Sarason's theorem stated above, 
we may write $g=g_1g_2$ with $g_1, g_2\in H^2(\mathbb{T};S^2)$ 
and $\norm{g_1}_{H^2(\T;S^2)}\norm{g_2}_{H^2(\T;S^2)}=
\norm{g}_{H^1(\T;S^1)}$. Let 
$$
h_1=(G_2\overline{\otimes}I_{S^2})(g_1)
\qquad\hbox{and}\qquad
h_2=(G_2\overline{\otimes}I_{S^2})(g_2).
$$
Then,
$\norm{h_1}_{H^2(\R;S^2)}\norm{h_2}_{H^2(\R;S^2)}=
\norm{h}_{H^1(\R;S^1)}$.
Furthermore the definitions of $G_1$ and $G_2$ ensure  that
$$
(G_2\overline{\otimes}I_{S^2})(g_1) (G_2\overline{\otimes}I_{S^2})(g_2) =
(G_1\overline{\otimes}I_{S^1})(g_1g_2),
$$
and hence $h=h_1h_2$.
\end{proof}

\section{A characterization of $S^1$-bounded Fourier multipliers on $H^1(\R)$}
\label{Multipliers}

Let $T\in B(H^1(\R))$. It is shown in \cite[Section 2]{AL} that $T$ commutes with translations
if and only if there exists $m\in L^\infty(\R_+)$ such that 
\begin{equation}\label{Symbol}
\widehat{T(h)} = m \widehat{h},\qquad h\in H^1(\R).
\end{equation}
Such operators are called bounded Fourier multipliers on $H^1(\R)$ and the set
of all bounded Fourier multipliers on $H^1(\R)$ is denoted by  ${\mathcal M}(H^1(\R))$.
For any $T\in  {\mathcal M}(H^1(\R))$, the function $m$ satisfying (\ref{Symbol})
is unique.
In this situation, we set $T=T_m$ and $m$ is called the symbol of $T_m$.
According to \cite[Theorem 2.2]{AL}, the symbol $m$ of a bounded Fourier 
multiplier $T_m$ on $H^1(\R)$
belongs to $C_b(\R_+^*)$, and $\norm{m}_\infty\leq\norm{T_m}$.

Let $Z$ be a Banach space and recall that by definition, $H^1(\R)\otimes Z$ is dense in $H^1(\R;Z)$.
We say that a bounded Fourier multiplier $T\colon H^1(\R)\to H^1(\R)$ is $Z$-bounded if the tensor extension
$T\otimes I_Z\colon H^1(\R)\otimes Z \to H^1(\R)\otimes Z$ extends to a bounded operator
$$
T\overline{\otimes} I_Z\colon H^1(\R;Z)\longrightarrow H^1(\R;Z).
$$

Following \cite{AL}, for any $\mu\in M(\R)$, we let $R_\mu\colon H^1(\R)\to H^1(\R)$ be the restriction of the 
convolution operator $h\mapsto \mu\ast h$ to $H^1(\R)$. This is a bounded Fourier
multiplier. Then we set 
\begin{equation}\label{Trivial}
{\mathcal R} = \{ R_\mu\, :\, \mu\in M(\R)\}\,\subset \MH.
\end{equation}
It is proved in \cite[Proposition 2.4]{AL} that 
a bounded Fourier multiplier $T\colon H^1(\R)\to H^1(\R)$ is $Z$-bounded
for any Banach space $Z$ if and only if $T\in{\mathcal R}$. Moreover for any
$\mu\in  M(\R)$, we have 
\begin{equation}\label{Z}
\bignorm{R_\mu\overline{\otimes} I_Z}\leq\norm{\mu}_{M(\R)}.
\end{equation}

The main result of this section is a characterization 
of $S^1$-bounded Fourier multipliers on $H^1(\R)$.
This is a real line analog of \cite[Theorem 6.2]{Pis1}.

\begin{thm}\label{Main}
Let $m\in L^\infty(\R_+)$  and let $C\geq0$. The following assertions are equivalent.
\begin{itemize}
\item [(i)] The function $m$ is the symbol of an $S^1$-bounded Fourier multiplier 
$T_m\colon H^1(\mathbb{R})\to H^1(\mathbb{R})$, with $\norm{T_m\overline{\otimes} I_{S^1}} \leq C$.  
\item [(ii)] There exist a Hilbert space $\mathcal{H}$ as well as  two
functions $\alpha,\beta \in L^\infty(\mathbb{R}_+;\mathcal{H})$ 
such that $\|\alpha\|_\infty\|\beta\|_\infty\leq C$ and  
$$
m(s+t)=\langle\alpha(t),\beta(s)\rangle_{\mathcal H}, 
$$
for almost every $(s,t)\in\mathbb{R}_{+}\times \mathbb{R}_{+}$.
\end{itemize}
\end{thm}

Condition (ii) means that the function $(s,t)\mapsto m(s+t)$ belongs to the 
space ${\mathcal V}_2(\mathbb{R}_{+}\times \mathbb{R}_{+})$ defined in Subsection \ref{Gamma2}.

The statement of Theorem \ref{Main}
and its proof are closely related to the theory of completely bounded maps.
We now supply some background on this topic and refer to Remark \ref{CBH1} for more on this
aspect. Let $A$ and $B$ be $C^*$-algebras. A bounded map
$w\colon A\to B$ is called completely bounded if there exists a constant $C\geq0$ such that 
$$
\forall\, n\geq 1,\qquad \bignorm{w\otimes I_{M_n}\colon M_n(A)\longrightarrow M_n(B)}\,\leq C.
$$
In this case, the smallest possible $C\geq0$ is called
the completely bounded norm of $w$ and is denoted by $\norm{w}_{cb}$.  In the above definition,
$M_n(A)$ (resp. $M_n(B)$) is the unique $C^*$-algebra built upon
$A\otimes M_n$ (resp. $B\otimes M_n$),  
whose product and involution 
are compatible with those of $A$ and $M_n$ (resp. of  $B$ and $M_n$). We refer the
reader to \cite{Pau} for more explanations and basics on complete boundedness.
We will need the following fundamental result, known as the Haagerup-Paulsen-Wittstock theorem.

\begin{thm}\label{Wittstock} ${\rm (}$\cite[Theorem 8.4]{Pau}${\rm )}$
Let $K$ be a Hilbert space, let $A$ be a $C^*$-algebra, let $w\colon A\to B(K)$
be a bounded map and let $C\geq 0$ be a constant. The following assertions are equivalent.
\begin{itemize}
\item [(i)] The map $w$ is completely bounded and $\norm{w}_{cb}\leq C$.
\item [(ii)] There exist a Hilbert space $\mathcal{H}$, a non-degenerate
$*$-representation $\pi\colon A\to B(\mathcal{H})$  as well as  two operators $V_1,V_2\colon K\to \mathcal{H}$ 
such that
$$
w(a) \,=\, V_2^* \pi(a) V_1,\qquad a\in A.
$$
\end{itemize}
\end{thm}

In the following proof we will use the fact that 
for any Hilbert space $H$,
the Fourier transform $L^2(\R;H)\to L^2(\R;H)$
induces an isomorphism from
$H^2(\R;H)$ onto $L^2(\R_+;H)$.

\begin{proof}[Proof of Theorem \ref{Main}]
\ 

$(ii)\,\Rightarrow\,(i):$  Assume (ii). 
We may and do assume that  $\|\alpha\|_\infty=1$ and $\|\beta\|_\infty=1$.

Since $\mathbb{R}_+$ is $\sigma$-finite and 
$\alpha\colon\R_+\to\mathcal{H}$ is measurable,
we may assume, by \cite[Theorem II.1.2]{DU},
that $\mathcal{H}$ is separable. We can also assume that it is infinite dimensional. 
Let $(e_n)_{n\geq1}$ be an orthonormal basis 
of $\mathcal{H}$. We may write
\[
m(s+t)=\mathop{\sum}\limits_{n=1}^\infty\alpha_n(t)\beta_n(s)
 \qquad\hbox{for\ 
a.e.-}(s,t)\in\mathbb{R}_{+}^2,
\]
where $\alpha_n=\langle\alpha(\cdot),e_n\rangle_{\mathcal{H}}$ 
and $\beta_n=\langle e_n,\beta(\cdot)\rangle_{\mathcal{H}}$
for all $n\geq 1$. 

Suppose that $h\in H^1(\mathbb{R};S^1)$ with $\|h\|_1=1$. By Proposition \ref{Sarason}, 
there exist $h_1,h_2\in H^2(\mathbb{R};S^2)$ with $\|h_1\|_2=1$, $\|h_2\|_2=1$,
such that $h=h_1h_2$. For any $n\geq1$, $\alpha_n$ and $\beta_n$
belong to $L^\infty(\R_+)$, hence $\alpha_n\widehat{h_1}$ and $ \beta_n\widehat{h_2}$ belong to
$L^2(\R_+;S^2)$. Consequently, there exist $h_{1,n}, h_{2,n}\in H^2(\mathbb{R};S^2)$ such that
\[
\widehat{h_{1,n}}=\alpha_n\widehat{h_1}\qquad\text{and}\qquad\widehat{h_{2,n}}=\beta_n\widehat{h_2}.
\]
Using the Cauchy-Schwarz inequality, we have  
\begin{equation}\label{CS}
\mathop{\sum}\limits_{n=1}^\infty\|h_{1,n}\|_2\|h_{2,n}\|_2
\leq\Bigl(\mathop{\sum}_{n=1}^{\infty}\|h_{1,n}\|_2^2\Bigr)^{\frac12}
\Bigl(\mathop{\sum}_{n=1}^{\infty}\|h_{2,n}\|_2^2\Bigr)^{\frac12}.
\end{equation}
Moreover, using (\ref{Plancherel}), we have
\begin{align*}
\mathop{\sum}_{n=1}^{\infty}\|h_{1,n}\|_2^2
&=\,\frac{1}{2\pi}\,\int_0^\infty \sum_{n=1}^{\infty} \lvert\alpha_n(t)\rvert^2\|\widehat{h_1}(t)\|_2^2\ dt\\
&=\,\frac{1}{2\pi}\,\int_{0}^\infty\|\alpha(t)\|^2_{\mathcal{H}}\|\widehat{h_1}(t)\|_2^2\ dt\\
& \leq \,\frac{1}{2\pi}\,\|\alpha\|_{\infty}^2  \norm{\widehat{h_1}}_2^2  =\|\alpha\|_{\infty}^2\|h_1\|_2^2=1.
\end{align*}
With a similar argument, we have $\mathop{\sum}\limits_{n=1}^\infty\|h_{2,n}\|_2^2\leq1$. Hence using \eqref{CS}, we obtain that
the product functions $h_{1,n}h_{2,n} \in H^1(\R;S^1)$ satisfy
\begin{equation}\label{Norm1}
\mathop{\sum}\limits_{n=1}^\infty\|h_{1,n}h_{2,n}\|_1\leq\mathop{\sum}\limits_{n=1}^\infty\|h_{1,n}\|_2\|h_{2,n}\|_2\leq1.
\end{equation}
Therefore, 
$$
\Phi :=\mathop{\sum}\limits_{n=1}^\infty h_{1,n}h_{2,n}
$$
is a well-defined element of $H^1(\mathbb{R};S^1)$.
Furthermore,  we have $\widehat{\Phi}=\mathop{\sum}_{n=1}^\infty\widehat{h_{1,n}h_{2,n}}$,
the series being convergent in $C_0(\mathbb{R};S^1)$. 

Consider $u>0$.
For every $n\geq1$, we have
\begin{align*}
\widehat{h_{1,n}h_{2,n}}(u)&=\int_0^u\widehat{h_{1,n}}(u-s)\widehat{h_{2,n}}(s)\, ds\\
&=\int_0^u\alpha_n(u-s)\beta_n(s)\widehat{h_1}(u-s)\widehat{h_2}(s)\, ds.
\end{align*}
Thus we have
\[
\widehat{\Phi}(u)=\mathop{\sum}\limits_{n=1}^{\infty}\int_0^u\alpha_n(u-s)\beta_n(s)\widehat{h_1}(u-s)\widehat{h_2}(s)\, ds.
\]
The summations $\mathop{\sum}\limits_{n=1}^{\infty}$ and $\int_0^u$ can be switched. Indeed arguing as above, we have
\begin{align*}
\int_0^u \sum_{n=1}^\infty  
\bignorm{\alpha_n(u-s) & \beta_n(s)  \widehat{h_1}(u-s)\widehat{h_2}(s)}_{S^1}\,  ds\\
&\leq\Bigl(\int_0^u\mathop{\sum}\limits_{n=1}^{\infty}
\lvert\alpha_n(u-s)\rvert^2\|\widehat{h_1}(u-s)\|_{S^2}^2 \,ds\Bigr)^{\frac{1}{2}}
\Bigl(\int_0^u\mathop{\sum}\limits_{n=1}^{\infty}
\lvert\beta_n(s)\rvert^2\|\widehat{h_2}(s)\|_{S^2}^2\,ds\Bigr)^{\frac{1}{2}}\\
&\leq\Bigl(\int_0^\infty\mathop{\sum}\limits_{n=1}^{\infty}\lvert\alpha_n(s')\rvert^2
\|\widehat{h_1}(s')\|_{S^2}^2\, ds'\Bigr)^{\frac{1}{2}}
\Bigl(\int_0^\infty\mathop{\sum}\limits_{n=1}^{\infty} \lvert\beta_n(s)\rvert^2 \|\widehat{h_2}(s)\|_{S^2}^2\,ds\Bigr)^{\frac{1}{2}}\\
&\leq \norm{\alpha}_\infty
\norm{\beta}_\infty\norm{\widehat{h_1}}_2 \norm{\widehat{h_2}}_2\, <\infty.\\
\end{align*}
Therefore, for all $u>0$, we have
$$
\widehat{\Phi}(u)=\displaystyle{\int_0^u}
\mathop{\sum}\limits_{n=1}^\infty 
\alpha_n(u-s)\beta_n(s) \widehat{h_1}(u-s)\widehat{h_2}(s) \, ds.
$$
For almost every $u>0$, we have 
$m(u) =\mathop{\sum}\limits_{n=1}^{\infty} \alpha_n(u-s)\beta_n(s)$, 
for almost every $s\in(0,u)$. Hence we obtain that
\[
\widehat{\Phi}(u)=
m(u)\int_0^u\widehat{h_1}(u-s)\widehat{h_2}(s)\, ds\,=m(u)\widehat{h}(u)
\]
for almost every $u>0$.

By  (\ref{Norm1}), we have $\|\Phi\|_1\leq1$.
Hence the equality $\widehat{\Phi} = m\widehat{h}$ shows that $m$ is the symbol of an
$S^1$-bounded Fourier multiplier $T_m$ and 
that $\norm{T_m\overline{\otimes} I_{S^1}}\leq1$.

\smallskip
$(i)\,\Rightarrow\,(ii):$
We use the classical duality pairing between $L^1(\R)$ and $L^\infty(\R)$ leading to $L^1(\R)^*\simeq L^\infty(\R)$.
Then we  have $H^1(\R)^\perp = H^\infty(\R)$, hence $H^1(\R)^*\simeq
\frac{L^\infty(\R)}{H^\infty(\R)}$.
According to \cite[Section IV.1]{DU}, the above extends to  $L^1(\R; S^1_n)^*\simeq L^\infty(\R;M_n)$
for all $n\geq 1$.  Hence we have  an isometric identification
\begin{equation}\label{H1Hinf}
H^1(\R;S^1_n)^*\,\simeq\, \frac{L^\infty(\R;M_n)}{H^\infty(\R;M_n)}
\end{equation}
for all $n\geq 1$.

Let $Q:L^\infty(\mathbb{R})\to B(H^2(\mathbb{R}))$ be given by
\[
\langle Q(f)h_1,h_2\rangle_{H^2(\R)}
=\int f(t) h_1(t)\overline{h_2}(-t)\, dt, \qquad f\in L^\infty(\R),\, h_1,h_2\in H^2(\mathbb{R}).
\]
Observe that $Q$ is completely bounded, with $\norm{Q}_{cb}\leq 1$. 
Indeed, let $\rho\colon L^\infty(\mathbb{R})\to B(L^2(\mathbb{R}))$ be the $*$-representation given by $\rho(f)h= fh$,
for $f\in L^\infty(\R)$ and $h\in L^2(\mathbb{R})$. Let $W_1\colon
H^2(\mathbb{R})\to L^2(\mathbb{R})$ be the inclusion map and let $W_2\colon
H^2(\mathbb{R})\to L^2(\mathbb{R})$ be defined by $W_2(h)=h(-\,\cdotp)$ for $h\in H^2(\mathbb{R})$. 
Then $Q=W_2^*\rho(\,\cdotp)W_1$. Hence the result follows from the easy 
implication ``$(ii)\Rightarrow(i)$" of Theorem \ref{Wittstock}.

The mapping $h\mapsto \overline{h}(-\cdotp)$ is a unitary of $H^2(\R)$. Hence
by the scalar case of Proposition \ref{Sarason}, the set $\{h_1\overline{h_2}(-\,\cdotp)\, :\, h_1,\,h_2\in H^2(\R)\}$ is equal 
to $H^1(\R)$. Therefore we have ${\rm Ker}(Q)=H^\infty(\mathbb{R})$. We let 
$$
\widetilde{Q}\colon \frac{L^\infty(\mathbb{R})}{H^\infty(\mathbb{R})}\longrightarrow B(H^2(\mathbb{R}))
$$
be induced by $Q$. Then it follows from above that
\begin{equation}\label{wQ-cb}
\Bignorm{\widetilde{Q}\otimes I_{M_n}\colon 
\frac{L^\infty(\mathbb{R};M_n)}{H^\infty(\mathbb{R};M_n)}\longrightarrow M_n\bigl(B(H^2(\mathbb{R}))\bigr)}\leq 1,\qquad n\geq 1.
\end{equation}

We note for further use that for any $a\in \frac{L^\infty(\mathbb{R})}{H^\infty(\mathbb{R})}$ and any $h_1,h_2\in H^2(\R)$, we have
\begin{equation}\label{Q}
\langle\widetilde{Q}(a) h_1,h_2\rangle_{H^2(\R)}\,=\,
\langle a, h_1\overline{h_2}(-\,\cdotp)\rangle_{\frac{L^\infty(\mathbb{R})}{H^\infty(\mathbb{R})}, H^1(\R)}.
\end{equation}

Assume (i) and consider the adjoint map 
$T_m^*\colon \frac{L^\infty(\mathbb{R})}{H^\infty(\mathbb{R})}\to \frac{L^\infty(\mathbb{R})}{H^\infty(\mathbb{R})}$. 
According to (\ref{H1Hinf}), we have
\begin{equation}\label{Tm-cb}
\Bignorm{T_m^*\otimes I_{M_n}\colon 
\frac{L^\infty(\mathbb{R};M_n)}{H^\infty(\mathbb{R};M_n)}\longrightarrow
\frac{L^\infty(\mathbb{R};M_n)}{H^\infty(\mathbb{R};M_n)}}\leq C,\qquad n\geq 1.
\end{equation}

For any $s\in\R$, we let $e_s\in L^\infty(\R)$ be defined by 
$e_s(t)=e^{ist}$. We will use the fact that $e_s\in H^\infty(\R)$ if (and only if)
$s\geq 0$. Let $\gamma\colon\mathbb{R}\to B(H^2(\mathbb{R}))$ be defined by 
$$
\gamma(s)=Q(e_{-s}),\qquad s\in\R.
$$
Let $q\colon L^\infty(\R)\to \frac{L^\infty(\R)}{H^\infty(\R)}$ denote the canonical quotient map. We claim that 
for all $s>0$, we have
\begin{equation}\label{m-gamma}
\widetilde{Q}T_m^*q(e_{-s})=m(s)\gamma(s).
\end{equation}
To check this, let $h_1,h_2\in H^2(\R)$ and let $h= h_1\overline{h_2}(-\,\cdotp)$. This function
belongs to $H^1(\R)$. Then we have
\begin{align*}
\bigl\langle\widetilde{Q}\bigl[ T_m^*q(e_{-s})\bigr]h_1,h_2\bigr\rangle_{H^2(\R)}
&=\langle T_m^*q(e_{-s}) ,h\rangle_{\frac{L^\infty(\mathbb{R})}{H^\infty(\mathbb{R})}, H^1(\R)}\\
&=\langle q(e_{-s}),T_m(h)\rangle_{\frac{L^\infty(\mathbb{R})}{H^\infty(\mathbb{R})},H^1(\mathbb{R})}\\
&=\widehat{T_m(h)}(s)\\
&=m(s)\widehat{h}(s)\\
&=m(s)\langle\gamma(s)h_1,h_2\rangle_{H^2(\R)},
\end{align*}
by (\ref{Q}) and the definition of $T_m$.
This proves (\ref{m-gamma}).

By (\ref{wQ-cb}) and (\ref{Tm-cb}), the mapping 
$\widetilde{Q} T_m^* q\colon L^\infty(\R)\to B(H^2(\R))$ is completely bounded, 
with $\norm{\widetilde{Q}T_m^*q}_{cb}\leq C$.
By Theorem \ref{Wittstock} applied to the restriction of $\widetilde{Q}T_m^*q$ to 
$C_0(\R)$, there exist a Hilbert space $\mathcal{H}$, a non-degenerate
$\ast$-representation $\pi\colon C_0(\mathbb{R})\to B(\mathcal{H})$ and two
operators  $V_1,V_2\colon H_2(\mathbb{R})\to\mathcal{H}$ such that 
$$
\widetilde{Q}T_m^*q_{\vert C_0(\R)}=V_2^*\pi(\,\cdotp)V_1
\qquad\hbox{and}\qquad
\norm{V_1}\norm{V_2}\leq C.
$$

Regard $L^1(\R)$ as a Banach algebra for convolution.
Since the Fourier transform $\mathcal{F}\colon L^1(\R)\to C_0(\R)$ is a contractive homomorphism
with dense range, the mapping $\pi\circ\mathcal{F}\colon L^1(\R)\to  B(\mathcal{H})$
is a contractive non-degenerate homomorphism. Hence by \cite[Proposition 13.3.4]{Dix}, there exists a strongly continuous unitary  representation
$\sigma\colon\R\to B(\mathcal{H})$ which ``generates" $\pi\circ\mathcal{F}$ in the following sense. For any
$g\in L^1(\R)$,
$$
\pi(\widehat{g}) = (\pi\circ\mathcal{F})(g) \,=\, \int_{-\infty}^\infty g(t)\sigma(t)\,dt,
$$
where the integral is defined in the strong operator topology. We claim that for all $s\in\R$, we have
\begin{equation}\label{sigma}
\widetilde{Q} T_m^* q(e_{-s}) \,=\, V_2^* \sigma(s)V_1.
\end{equation}
To prove this, we fix $s\in\R$, we introduce a function $\eta \in L^1(\R)$ with $\int_\R\eta=1$, and 
we consider the approximate
unit $(\eta_\lambda)_{\lambda>0}$ defined by $\eta_\lambda(t)=\lambda\eta(\lambda t)$ 
for all $\lambda>0$ and $t\in\R$. On the one hand, we note that $\widehat{\eta_\lambda}\to 1$ pointwise, hence
by Lebesgue's dominated convergence theorem, $e_{-s}\widehat{\eta_\lambda}(-\,\cdotp)\to e_{-s}$ in the $w^*$-topology 
of $L^\infty(\R)$, when $\lambda\to\infty$. By construction,
$\widetilde{Q} T_m^* q\colon L^\infty(\R)\to B(H^2(\R))$ is $w^*$-continuous. 
Hence we deduce that
$$
\widetilde{Q} T_m^* q(e_{-s}) \,=\, w^*-\lim_{\lambda\to\infty} 
\widetilde{Q} T_m^* q\bigl(e_{-s}\widehat{\eta_\lambda}(-\,\cdotp)\bigr).
$$
On the other hand, take
$\xi_1,\xi_2\in\mathcal{H}$. The function
$t\mapsto\langle\sigma(t)\xi_1,\xi_2\rangle$ is continuous hence
$$
\int_{-\infty}^\infty \eta_\lambda (s-t)\langle\sigma(t)\xi_1,\xi_2\rangle_{\mathcal H}
\, dt\,
\overset{\lambda\to\infty}{\longrightarrow}\,
\langle\sigma(s)\xi_1,\xi_2\rangle_{\mathcal H},
$$
which means that
$$
\bigl\langle (\pi\circ\mathcal{F})\bigl(\eta_\lambda(s-\,\cdot)\bigr)\xi_1,\xi_2\bigr\rangle_{\mathcal H}
\overset{\lambda\to\infty}{\longrightarrow}\,
\langle\sigma(s)\xi_1,\xi_2\rangle_{\mathcal H}.
$$
Since $\bignorm{(\pi\circ\mathcal{F})\bigl(\eta_\lambda(s-\,\cdot)\bigr)}
\leq\norm{\eta}_1$
for any $\lambda>0$, this
implies that $(\pi\circ\mathcal{F})\bigl(\eta_\lambda(s-\,\cdot)\bigr)\to \sigma(s)$
in the $w^*$-topology of $B(\mathcal{H})$, when $\lambda\to\infty$. Consequently,
$$
V_2^*\sigma(s)V_1
\,=\,w^*-\lim_{\lambda\to\infty} 
V_2^*\bigl[\pi\circ\mathcal{F}\bigl(\eta_\lambda(s-\,\cdotp)\bigr)\bigr]V_1.
$$
Since $\mathcal{F}\bigl(\eta_\lambda(s-\,\cdotp)\bigr)=e_{-s}\widehat{\eta_\lambda}(-\,\cdotp)$,
the latter belongs to $C_0(\R)$ and we actually obtain that 
$$
V_2^*\sigma(s)V_1
\,=\, w^*-\lim_{\lambda\to\infty} 
\widetilde{Q} T_m^* q\bigl(e_{-s}\widehat{\eta_\lambda}(-\,\cdotp)\bigr).
$$
The identity (\ref{sigma}) follows at once.

Next we show that for any $\epsilon>0$, the function $m_{\epsilon}
\colon \mathbb{R}_+\to\mathbb{C}$ given by 
\[
m_{\epsilon}(t)=m(\epsilon+t),\qquad t\geq 0,
\]
satisfies the condition (ii) of Theorem \ref{Main}.

Let $\chi_\epsilon$ be the indicator function
of the interval $(0,\epsilon)$. Then $\|\epsilon^{-\frac12}\chi_{\epsilon}\|_2=1$ 
and $\mathcal{F}^{-1}\bigl(\epsilon^{-\frac12 }\chi_{\epsilon}\bigr)$ belongs to $H^2(\mathbb{R})$. 
We set $\zeta_{\epsilon} = \sqrt{2\pi}\,e_{-\frac{\epsilon}{2}}\mathcal{F}^{-1}\bigl(\epsilon^{-\frac12 }\chi_\epsilon\bigr)$. 
We have
\begin{align*}
\int_{-\infty}^{\infty} \zeta_{\epsilon}(u)\overline{\zeta_{\epsilon}(-u)}\,du\, &=\,
\frac{2\pi}{\epsilon}\int_{-\infty}^{\infty} e^{-i\epsilon u}
\bigl[\mathcal{F}^{-1}(\chi_{\epsilon})\mathcal{F}^{-1}(\chi_{\epsilon})\Bigr](u)\, du\\
&=\,\frac{2\pi}{\epsilon}\mathcal{F}[\mathcal{F}^{-1}(\chi_{\epsilon})\mathcal{F}^{-1}(\chi_{\epsilon})](\epsilon)\\
&=\frac{1}{\epsilon}\left(\chi_{\epsilon}\ast \chi_{\epsilon}\right)(\epsilon)=1.
\end{align*}

By construction, $e_{\frac{\epsilon}{2}}\zeta_{\epsilon}\in H^2(\mathbb{R})$
hence  $e_{t+\frac{\epsilon}{2}}\zeta_{\epsilon}\in H^2(\mathbb{R})$, for any $t\geq 0$.
Moreover by (\ref{Plancherel}), we have
\begin{equation}\label{=1}
\norm{e_{t+\frac{\epsilon}{2}}\zeta_{\epsilon}}_2 = 
\|\zeta_{\epsilon}\|_2=1,\qquad t\geq 0.
\end{equation}

Then we may  define $\alpha_{\epsilon}\colon\R_+\to\mathcal{H}$ and 
$\beta_{\epsilon}\colon\R_+\to\mathcal{H}$ by
\[
\alpha_{\epsilon}(t)=\sigma(t+\tfrac{\epsilon}{2}) V_1(e_{t+\frac{\epsilon}{2}}\zeta_{\epsilon})
\qquad\text{and}\qquad
\beta_{\epsilon}(s)=\sigma(-s-\tfrac{\epsilon}{2})V_2(e_{s+\frac{\epsilon}{2}}\zeta_{\epsilon}).
\]
Since $\sigma$ is a unitary representation, we have, using (\ref{m-gamma}) and (\ref{sigma}),
\begin{align*}
\langle \alpha_{\epsilon}(t),\beta_{\epsilon}(s)\rangle_{\mathcal H}
&=\langle\sigma(t+\tfrac{\epsilon}{2})V_1(e_{t+\frac{\epsilon}{2}} \zeta_{\epsilon}),
\sigma(-s-\tfrac{\epsilon}{2}) V_2(e_{s+\frac{\epsilon}{2}}\zeta_{\epsilon})\rangle_{\mathcal H}\\
&=\langle V_2^*\sigma(-s-\tfrac{\epsilon}{2})^* \sigma(t+\tfrac{\epsilon}{2}) 
V_1(e_{t+\frac{\epsilon}{2}}\zeta_{\epsilon}),e_{s+\frac{\epsilon}{2}}\zeta_\epsilon\rangle_{\mathcal H}\\
&= \langle V_2^*\sigma(s+t+ \epsilon) 
V_1 (e_{t+\frac{\epsilon}{2}}\zeta_{\epsilon}) ,e_{s+\frac{\epsilon}{2}}\zeta_\epsilon\rangle_{\mathcal H}\\
&=m(s+t+\epsilon)\langle\gamma(s+t+\epsilon) e_{t+\frac{\epsilon}{2}}\zeta_{\epsilon} ,
e_{s+\frac{\epsilon}{2}}\zeta_{\epsilon}\rangle_{\mathcal H}.
\end{align*}
Further
\begin{align*}
\langle\gamma(s+t+\epsilon) e_{t+\frac{\epsilon}{2}}\zeta_{\epsilon} ,e_{s+\frac{\epsilon}{2}}\zeta_{\epsilon}\rangle_{\mathcal H}
& = \int_{-\infty}^{\infty} e^{-i(t+s+\epsilon)u} e^{i(t+\frac{\epsilon}{2})u}\zeta_{\epsilon}(u) 
e^{i(s+\frac{\epsilon}{2})u}\overline{\zeta_{\epsilon}}(-u)\, du\\
& = \int_{-\infty}^{\infty} \zeta_{\epsilon}(u)\overline{\zeta_{\epsilon}(-u)}\,du =1
\end{align*}
by the above computation. Thus we obtain that
$$
m_{\epsilon}(t+s)\,=\,\langle \alpha_{\epsilon}(t),\beta_{\epsilon}(s)\rangle_{\mathcal H},\qquad s,t \geq 0.
$$
The function  $\alpha_\epsilon$ is continuous, hence measurable. Moreover by (\ref{=1}), we
have $\norm{\alpha_\epsilon(t)}\leq\norm{V_1}$
for all $t\geq 0$. Thus we have $\alpha_\epsilon\in L^\infty(\R_+;\mathcal{H})$
with $\norm{\alpha_\epsilon}_\infty\leq \norm{V_1}$. 
Likewise  $\beta_\epsilon \in L^\infty(\R_+;\mathcal{H})$
with $\norm{\beta_\epsilon}_\infty\leq \norm{V_2}$.
Hence $\norm{\alpha_\epsilon}_\infty\norm{\beta_\epsilon}_\infty\leq C$.

We define $\varphi_\epsilon$ and $\varphi$ in $L^\infty(\R_+^2)$ by
$$
\varphi(s,t) = m(s+t)
\quad\text{and}\quad \varphi_\epsilon(s,t) = m_\epsilon(s+t),\qquad s,t >0.
$$
With the terminology introduced in Subsection \ref{Gamma2}, we have just proved that $\varphi_\epsilon\in\mathcal{V}_2(\R_+^2)$
and that $\nu_2(\varphi_\epsilon)\leq C$ for all $\epsilon>0$. 
Since $m\colon\R_+^*\to\C\,$ 
is continuous, by \cite[Theorem 2.2]{AL}, $m_\epsilon\to m$ pointwise on $\R_+^*$. Since $m$ 
is bounded, we deduce by Lebesgue's dominated convergence theorem  that $\varphi_\epsilon\to\varphi$
in the $w^*$-topology of  $L^\infty(\R_+^2)$. Applying Lemma \ref{V2} (2), we deduce that
$\varphi$ belongs to $\mathcal{V}_2(\R_+^2)$, with $\nu_2(\varphi)\leq C$. 
This means that $m$ satisfies (ii).
\end{proof}

\begin{rq1}\label{Q-Cont}
Let $m\in L^\infty(\R_+)$ be the symbol of an $S^1$-bounded Fourier multiplier on 
$S^1(\R)$ and for any $\epsilon>0$, let $m_\varepsilon\in  L^\infty(\R_+)$ be defined
by $m_\epsilon(t)=m(\epsilon +t)$. We know from \cite[Theorem 2.2]{AL} that $m\in C_b(\R_+^*)$. 
Furthermore it follows
from the above proof that for any $\epsilon>0$, there exist a Hilbert space ${\mathcal H}$
and continuous bounded functions
$\alpha,\beta\colon\R_+^*\to{\mathcal H}$ such that $m_\epsilon(s+t) =\langle \alpha(t),\beta(s)\rangle_{\mathcal H}$
for all $s,t>0$. However we do not know if the assertion (ii) of Theorem \ref{Main}
holds true for some continuous bounded functions 
$\alpha$ and $\beta$ from $\R_+^*$ into ${\mathcal H}$.
\end{rq1}

\begin{rq1}\label{CBH1}
This remark is about operator space theory, for which we refer e.g.  either to
\cite{ER} for a comprehensive treatment, or to  \cite[Chapter 1]{BLM} for a summary.
In this framework,  $H^1(\R)$ has a natural operator space structure. It is obtained
by regarding  $H^1(\R)$ as the sub-operator space of $L^1(\R)$ equipped with its so-called 
maximal operator space structure \cite[1.2.22]{BLM}.
Assume that 
$\frac{L^\infty(\R)}{H^\infty(\R)}$ is equipped with the quotient operator 
space structure induced by the $C^*$-algebra structure of $L^\infty(\R)$ \cite[1.2.14]{BLM}.
Then it follows from the operator space duality
of subspaces and quotients \cite[1.4.4]{BLM} that the identity
$$
H^1(\R)^*\simeq
\frac{L^\infty(\R)}{H^\infty(\R)}
$$
holds completely isometrically. Consequently, a bounded map $T\colon H^1(\R)\to H^1(\R)$
is completely bounded if and only if $T^*\colon \frac{L^\infty(\R)}{H^\infty(\R)}\to \frac{L^\infty(\R)}{H^\infty(\R)}$
is completely bounded, and $\norm{T}_{cb}=\norm{T^*}_{cb}$ in this case \cite[1.4.3]{BLM}.
It therefore follows from (\ref{H1Hinf}) that $T$ is $S^1$-bounded if
and only if it is completely bounded, and that $\norm{T\overline{\otimes} I_{S^1}}=\norm{T}_{cb}$
in this case.

Thus Theorem \ref{Main} actually provides a characterization of completely bounded 
Fourier multipliers on $H^1(\R)$. 
\end{rq1}

\begin{rq1}\label{Not-cb}
Let $r\in\R^*$ and let $m\colon \R_+^*\to\C$ be defined by
$$
m(t) = t^{ir},\qquad t>0.
$$
It is well-known that $m$ is the symbol of a bounded Fourier multiplier 
on $H^1(\R)$. A proof of this fact is obtained by applying
\cite[Lemma 4.1]{L} with $F(z)=z^{ir}$ together with
\cite[Theorem 1.2]{L} with $X=\C$. However $m$ is not 
$S^1$-bounded. This follows from \cite[Theorem 1.2]{L} 
and the fact that $S^1$ does not satisfy the so-called analytic
unconditional martingale difference (AUMD) property \cite[Theorem 4.1]{HP}.

We will see in Remark \ref{CB-Reg} that there exist $S^1$-bounded
multipliers on $H^1(\R)$ which do not belong to the space $\mathcal R$ defined by (\ref{Trivial}).
\end{rq1}

\section{The algebra $\A_{0,S^1}$ and its duality properties}
\label{A0}

From now on we let ${\mathcal M}_{S^1}(H^1(\R))$ denote the space of all 
$S^1$-bounded Fourier multipliers $T\colon H^1(\R)\to H^1(\R)$, equipped
with the norm
$$
\bignorm{T\overline{\otimes} I_{S^1}\colon H^1(\R;S^1)
\to H^1(\R;S^1)}.
$$
It is easy to check that ${\mathcal M}_{S^1}(H^1(\R))$ is a Banach space.
By construction, the mapping $T\mapsto T\overline{\otimes} I_{S^1}$
provides an isometric embedding
$$
{\mathcal M}_{S^1}(H^1(\R))\subset B\bigl(H^1(\R;S^1)\bigr).
$$
In this section we introduce a Banach algebra, denoted
by $\A_{0,S^1}(\R)$, whose dual space is isometrically isomorphic
to ${\mathcal M}_{S^1}(H^1(\R))$. This will be fully exploited 
in Section \ref{FC}.

The space $\A_{0,S^1}(\R)$
is a variant of the Banach algebra $\A_{0}(\R)$ introduced 
in \cite{AL}. Some of the results and arguments below are similar to the ones in 
\cite[Subsections 3.1--3.3]{AL} so we will allow ourselves not to
write all the details.

For any $f\in C_0(\R;S^\infty)$ and $h\in H^1(\R;S^1)$, we let $f\ast h\in C_0(\R)$
be the function obtained by mixing convolution 
and $(S^\infty, S^1)$-duality as follows:
$$
(f\ast h)(s) =\int_{-\infty}^{\infty} tr\bigl(f(t)h(s-t)\bigr)\, dt,\qquad
s\in\R.
$$
It is clear that $\norm{f\ast h}_\infty\leq\norm{f}_\infty\norm{h}_1$,
where $\norm{f}_\infty$ and $\norm{h}_1$ denote the norm
of $f$ in $C_0(\R;S^\infty)$ and the norm of $h$ in $H^1(\R;S^1)$, respectively.

We let $\mathcal{A}_{0,S^1}(\R)$ 
be the space of all functions $F
\colon \R \rightarrow \C$ such that there 
exist two sequences $(f_k)_{k\geq 1}$ in 
$C_0(\R;S^\infty)$  
and $(h_k)_{k\geq 1}$  in $H^1(\R;S^1)$ satisfying
\begin{equation}\label{DefA0}
\sum_{k=1}^{\infty} \normeinf{f_k}\norme{h_k}_1 < \infty
\qquad\hbox{and}\qquad
F(s) = \sum_{k=1}^{\infty} f_k\ast h_k(s),\quad s\in\R.
\end{equation}
For such a function $F$, we 
set
\begin{equation}\label{DefNormA0}
\norme{F}_{\A_{0,S^1}} = \inf \Bigl\{\sum_{k=1}^{\infty} 
\normeinf{f_k}\norme{h_k}_1 \Bigr\},
\end{equation}
where the infimum runs over all sequences  $(f_k)_{k\geq 1}$ in 
$C_0(\R;S^\infty)$  
and $(h_k)_{k\geq 1}$  in $H^1(\R;S^1)$ satisfying (\ref{DefA0}). 

It is clear that $\A_{0,S^1}(\R)\subset C_0(\R)$ and that 
\begin{equation}\label{Contraction}
\norme{F}_\infty\leq \norme{F}_{\A_{0,S^1}},\qquad 
F \in \A_{0,S^1}(\R).
\end{equation}
Further the argument in \cite[Subsection 3.1]{AL} showing
that $\A_{0}(\R)$ is a Banach space shows as well
that $\A_{0,S^1}(\R)$ is a Banach space.

\begin{prop}\label{BanAlg}
The space $\A_{0,S^1}(\R)$ is a Banach algebra for the pointwise multiplication.  
\end{prop}

\begin{proof}
We let $\ell^2\overset{2}\otimes\ell^2$ denote the Hilbertian tensor product
of two copies of $\ell^2$. Then we
let $S^\infty\otimes_{\rm min} S^\infty$ be the minimal tensor
product of two copies of $S^\infty$ and we recall the 
$*$-isomorphic identification 
$$
S^\infty\otimes_{\rm min} S^\infty\simeq 
S^\infty\bigl(\ell^2\overset{2}\otimes\ell^2\bigr),
$$
see e.g. \cite[Chapter 12]{Pau}. 
Similarly, regard $S^1\otimes S^1\subset S^1(\ell^2\overset{2}\otimes\ell^2)$ in 
the standard way. This is a dense subspace. Hence if we let $d_1$ denote the norm on $S^1\otimes S^1$ induced
by $S^1(\ell^2\overset{2}\otimes\ell^2)$ and if we let $S^1\otimes_{d_1} S^1$
denote the completion of $(S^1\otimes S^1, d_1)$, we have an isometric 
identification 
$$
S^1\otimes_{d_1} S^1  \simeq 
S^1(\ell^2\overset{2}\otimes\ell^2).
$$ 
(The reader familiar with operator
space theory will observe that $d_1$ is equal to the operator 
space projective tensor norm \cite[1.5.11]{BLM}.)

Let $f_1, f_2 \in C_0(\R;S^\infty)$ and $h_1,h_2 \in H^1(\R;S^1)$. 
We define, for any $s\in \R$, 
$$
\varphi_s\colon \R \longrightarrow 
S^\infty\otimes_{\rm min} S^\infty
\qquad\hbox{and}\qquad \psi_s 
\colon \R \longrightarrow S^1\otimes_{d_1} S^1
$$
by
\[
\varphi_s(t) = f_1(t)\otimes f_2(t-s) 
\qquad\hbox{and}\qquad
\psi_s(t) = h_1(t)\otimes h_2(t+s).
\]
Since $f_2$ is uniformly continuous, 
$s \mapsto \varphi_s$ is a continuous and bounded function from $\R$ into $C_0(\R;
S^\infty\otimes_{\rm min} S^\infty)$, hence an element of the 
space $L^{\infty}\bigl(\R; C_0(\R;S^\infty\otimes_{\rm min} S^\infty)\bigr)$. 
Since 
\begin{align*}
\int_{-\infty}^{\infty}\int_{-\infty}^{\infty}
\norm{h_1(t)\otimes h_2(t+s)}_{S^1\otimes_{d_1} S^1}\,dtds \, & =\,
\int_{-\infty}^{\infty}\int_{-\infty}^{\infty}\norm{h_1(t)}_{S^1}
\norm{h_2(t+s)}_{S^1}
\,dtds\\
& = \norme{h_1}_1\norme{h_2}_1,
\end{align*}
the function $s\mapsto \psi_s$ is defined almost everywhere
and belongs to $L^{1}\bigl(\R; L^1(\R;S^1\otimes_{d_1} S^1)\bigr)$.
The fact that $h_1,h_2$ belong to the analytic space $H^1(\R;S^1)$
ensures that $s\mapsto \psi_s$ actually
belongs to $L^{1}\bigl(\R; H^1(\R;S^1\otimes_{d_1}  S^1)\bigr)$.

Using natural identifications
$$
S^\infty\otimes_{\rm min} S^\infty\,\simeq S^\infty
\qquad\hbox{and}\qquad
S^1\otimes_{d_1} S^1\,\simeq S^1
$$
provided by $\ell^2\overset{2}\otimes\ell^2\simeq \ell^2$,
we see that for almost every $s\in \R$,
the function
$$
\varphi_s\ast \psi_s \colon u\mapsto\,
\int_{-\infty}^{\infty} tr\bigl(f_1(t)h_1(u-t)\bigr)
tr\bigl(f_2(t-s)h_2(u-t+s)\bigr)\, dt
$$
belongs to $\A_{0,S^1}(\R)$. Moreover 
the function $s\mapsto  \varphi_s\ast \psi_s$ 
belongs to $L^{1}\bigl(\R; \A_{0,S^1}(\R)\bigr)$, with
\begin{align*}
\int_{-\infty}^{\infty}\norme{\varphi_s\ast \psi_s}_{\A_{0,S^1}}ds 
&\leq \int_{-\infty}^{\infty} \norme{\varphi_s}_{\rm min}\norme{\psi_s}_{d_1} ds\\
&\leq \normeinf{f_1}\normeinf{f_2}\int_{-\infty}^{\infty}\norme{\psi_s}_{d_1} ds\\
&\leq \normeinf{f_1}\normeinf{f_2}\norme{h_1}_1\norme{h_2}_1.
\end{align*}
Now arguing as in the proof of \cite[Proposition 3.4]{AL}, we obtain that 
$$
\int_{-\infty}^{\infty}\varphi_s\ast \psi_s\, ds = (f_1\ast h_1)(f_2\ast h_2),
$$
and hence 
that  $(f_1\ast h_1)(f_2\ast h_2)$ belongs to
$\A_{0,S^1}(\R)$, with
$$
\norme{(f_1\ast h_1)(f_2\ast h_2)}_{\A_{0,S^1}} \leq 
\normeinf{f_1}\normeinf{f_2}\norme{h_1}_1\norme{h_2}_1.
$$
With this estimate in hands, the reasoning in  the proof of 
\cite[Proposition 3.4]{AL}
shows that $\A_{0,S^1}(\R)$ is a Banach algebra.
\end{proof}

Recall that the definition
of the Banach algebra  $\A_0(\R)$ introduced in 
\cite[Definitions 3.1 and 3.2]{AL}
is similar to the one of $\A_{0, S^1}(\R)$, except that 
the sequences $(f_k)_{k\geq 1}$ and $(h_k)_{k\geq 1}$ used in
(\ref{DefA0}) and (\ref{DefNormA0}) are requested to
belong to the scalar valued spaces 
$C_0(\R)$ and $H^1(\R)$, respectively. Thus we have 
$\A_{0}(\R)\subset \A_{0, S^1}(\R)$ and
\begin{equation}\label{A0A0S1}
\norm{F}_{\A_{0,S^1}}\leq 
\norm{F}_{\A_{0}},\qquad F\in\A_0(\R).
\end{equation}

We now turn to duality. We will use the fact that combining
(\ref{H1S1}) and (\ref{Proj}), we have an isometric identification
\begin{equation}\label{BH1-dual}
B\bigl(H^1(\R;S^1)\bigr)\simeq 
\biggl(\frac{C_0(\R;S^\infty)}{Z_0\overset{\vee}{\otimes} 
S^\infty}\,\widehat{\otimes}\, H^1(\R;S^1)\biggr)^*.
\end{equation}
Thus $B(H^1(\R;S^1))$ is a dual space.

For any $s\in\R$, we use the notation $\tau_s$ to denote the
translation operator $h\mapsto h(\,\cdotp -s)$ on 
$H^1(\R;S^1)$. 
Then for any $f\in C_0(\R;S^\infty)$
and $h\in H^1(\R;S^1)$, we have
$$
(f\ast h)(s) = \langle \tau_{-s}h, f\rangle,\qquad s\in\R,
$$
where the duality product on the right-hand side
is given by (\ref{Meziani}) and (\ref{Meziani2}).
By Proposition \ref{H1Dual2}, this implies that 
$f\ast h=0$ for any $f\in Z_0\overset{\vee}{\otimes} S^\infty$ and 
$h\in H^1(\R;S^1)$. Applying \cite[Section VIII.2, Theorem 1]{DU},
we deduce the existence of 
a necessarily unique contractive map
$$
\Lambda \colon 
\frac{C_0(\R;S^\infty)}{Z_0\overset{\vee}{\otimes} S^\infty}\, 
\widehat{\otimes}\, H^1(\R;S^1) \longrightarrow \A_{0,S^1}(\R)
$$
such that 
$$
\Lambda(\dot{f}\otimes h)\,=\,
f\ast h,\qquad f\in C_0(\R;S^\infty),\, h\in H^1(\R;S^1).
$$
Here $\dot{f}$ denotes the class of $f$ in 
$\frac{C_0(\R;S^\infty)}{Z_0\overset{\vee}{\otimes} S^\infty}$.
Moreover by the very definition of  $\A_{0,S^1}(\R)$, $\Lambda$ 
is a quotient map, in the sense of Definition \ref{Quotient}.

Regarding its adjoint as valued in
$B(H^1(\R;S^1))$, by (\ref{BH1-dual}), we obtain
that 
$$
\Lambda^*\colon \A_{0,S^1}(\R)^*\longrightarrow
B\bigl(H^1(\R;S^1)\bigr)
$$
is an isometry.

\begin{prop}\label{Range}
The range of $\Lambda^*$ is equal to ${\mathcal M}_{S^1}(H^1(\R))$.
\end{prop}

\begin{proof}
Since $\Lambda$ is a quotient map, ${\rm Ran}(\Lambda^*)={\rm Ker}(\Lambda)^\perp$
is $w^*$-closed.

Recall 
${\mathcal R}$ given by (\ref{Trivial}). Any element of ${\mathcal R}$ is $S^1$-bounded and we may therefore consider 
$$
{\mathcal R}_{S^1}=\,\bigl\{R_\mu\overline{\otimes} I_{S^1}\, :\, \mu\in M(\R)\bigr\}\,\subset 
B(H^1(\R;S^1)).
$$
Arguing as in \cite[Remark 3.8]{AL}, we see that for any $\mu\in M(\R)$ and any
$\Phi\in \frac{C_0(\R;S^\infty)}{Z_0\overset{\vee}{\otimes} S^\infty}\, 
\widehat{\otimes}\, H^1(\R;S^1)$, we have
\begin{equation}\label{DualDual}
\langle\mu,\Lambda(\Phi)\rangle\,=\, \bigl\langle R_\mu\overline{\otimes} I_{S^1}, \Phi\bigr\rangle,
\end{equation}
where the duality pairing in the left-hand side is given by (\ref{Riesz}) whereas 
the duality pairing in the right-hand side is given by (\ref{BH1-dual}).
This implies that 
$$
{\rm Ker}(\Lambda) =  ({\mathcal R}_{S^1})_\perp. 
$$
Consequently,
\begin{equation}\label{Interm}
{\rm Ran}(\Lambda^*) =\overline{{\mathcal R}_{S^1}}^{w^*}.
\end{equation}

Let us show that ${\mathcal M}_{S^1}(H^1(\R))$ is a $w^*$-closed 
subspace of $B(H^1(\R;S^1))$. By the Krein-Smulian theorem, it suffices to show that 
the unit ball of ${\mathcal M}_{S^1}(H^1(\R))$ is $w^*$-closed. To prove this,
let $(T_i)_i$ be a net of $S^1$-bounded Fourier multipliers on $H^1(\R)$ converging
to some $W\in B(H^1(\R;S^1))$ in the $w^*$-topology of $B(H^1(\R;S^1))$, and assume that 
$\norm{T_i\overline{\otimes} I_{S^1}}\leq 1$ for all $i$. 
This implies that $\norm{T_i}\leq 1$ for all $i$. Hence passing 
to a subnet, we may assume that  $(T_i)_i$ converges to some 
$T\in B(H^1(\R))$ in  the $w^*$-topology of $B(H^1(\R))$.
By \cite[Lemma 3.5]{AL}, $T$ is a bounded Fourier multiplier on $H^1(\R)$.
We observe that for any $h\in H^1(\R)$, $f\in C_0(\R)$, $a\in S^\infty$
and $b\in S^1$, we have 
$$
\bigl\langle   f\otimes a,  (T\otimes I_{S^1})(h\otimes b)\bigr\rangle
= \bigl\langle f\otimes a, W(h\otimes b)\bigr\rangle.
$$
Indeed both sides of this equality are equal to the limit of
$\langle f\otimes a, T_i(h)\otimes b\rangle$ when $i\to\infty$.
This implies that $T\otimes I_{S^1}$ and $W$ coincide on 
$H^1(\R)\otimes S^1$. Consequently,
$T$ is $S^1$-bounded and $T\overline{\otimes} I_{S^1}=W$.
Thus $W$ belongs to ${\mathcal M}_{S^1}(H^1(\R))$, which proves the $w^*$-closedness.

We have an inclusion ${\mathcal R}_{S^1}\subset {\mathcal M}_{S^1}(H^1(\R))$ by construction.
Hence to obtain the result that ${\rm Ran}(\Lambda^*)={\mathcal M}_{S^1}(H^1(\R))$,
it now suffices, by (\ref{Interm}), to show that
\begin{equation}\label{Haase}
{\mathcal M}_{S^1}(H^1(\R))\subset \overline{{\mathcal R}_{S^1}}^{w^*}.
\end{equation}
The remainder of the proof is devoted to this inclusion.

We let $\chi_+$ denote the characteristic function of $(0,\infty)$.
Let $\phi\in L^1(\R)$ such that $\widehat{\phi}\colon\R\to\C\,$ is a $C^\infty$-function with
compact support and $\widehat{\phi}(u)=1$
for all $u\in[-1,1]$. For any $n\geq 1$, define $\phi_n\in L^1(\R)$ by
$$
\phi_n(t) = n\phi(nt) - \,\frac{1}{n}\phi\Bigl(\frac{t}{n}\Bigr),\qquad t\in\R.
$$
Then $\widehat{\phi_n}(u) = \widehat{\phi}\bigl(\frac{u}{n}\bigr) -  \widehat{\phi}(nu)$
for all $u\in\R$, hence $\widehat{\phi_n}(u)=0$
for all $u\in\bigl[-\frac{1}{n},\frac{1}{n}\bigr]$. 
Consequently, the product $\chi_+\widehat{\phi_n}$
is a $C^\infty$-function with
compact support included in $\R_+^*$.
Hence there exists $\phi_{n,+}\in H^1(\R)$ such that 
$$
\widehat{\phi_{n,+}}\,  =\, \chi_+\widehat{\phi_n}.
$$
This construction is inspired by \cite[Section E.5]{Haase1}. It follows from 
\cite[Corollary E.5.3]{Haase1} that for any $h\in H^1(\R)$,
$\norm{\phi_n\ast h -h}_1\to 0$ when $n\to\infty$. 

Let $T\in {\mathcal M}_{S^1}(H^1(\R))$, with symbol $m$. For any 
$n\geq 1$, let $T_{(n)} = R_{\phi_n}\circ T$. We have 
$\norm{\phi_n}_1\leq 2\norm{\phi}_1$ hence 
$\norm{R_{\phi_n}\overline{\otimes} I_{S^1}}\leq 2\norm{\phi}_1$,
by (\ref{Z}). This yields
\begin{equation}\label{Unif}
\bignorm{T_{(n)}\overline{\otimes} I_{S^1}}\leq 2
\norm{\phi}_1 \bignorm{T\overline{\otimes} I_{S^1}},\qquad n\geq 1.
\end{equation}
For any $h\in H^1(\R)$ we have
$T_{(n)}(h) = \phi_{n}\ast T(h)$ hence
$$
\widehat{T_{(n)}(h)} = \widehat{\phi_n} \widehat{T(h)} =  \widehat{\phi_{n,+}}
m\widehat{h} = \widehat{T(\phi_{n,+})}\widehat{h}.
$$
This shows that $T_{(n)}$ belongs to  ${\mathcal R}$, with
$$
T_{(n)} = R_{T(\phi_{n,+})}, \qquad n\geq 1.
$$
We noticed that $T_{(n)}(h) \to T(h)$ in $H^1(\R)$ for any $h\in H^1(\R)$.
By linearity, this implies that   $(T_{(n)}\otimes I_{S^1})(h) \to 
(T\otimes I_{S^1})(h)$ in $H^1(\R;S^1)$ for any $h\in H^1(\R)\otimes S^1$.
According to the uniform estimate (\ref{Unif}), this implies that 
$(T_{(n)}\overline{\otimes} I_{S^1})(h) \to 
(T\overline{\otimes}  I_{S^1})(h)$ in $H^1(\R;S^1)$ for any $h\in H^1(\R;S^1)$.
A fortiori, we have $\bigl\langle (T_{(n)}\overline{\otimes} I_{S^1})(h),\dot{f}\bigr\rangle \to 
\bigl\langle (T\overline{\otimes} I_{S^1})(h),\dot{f}\bigr\rangle$ for any $h\in H^1(\R;S^1)$
and any $f\in C_0(\R,S^\infty)$. By linearity, this implies that through 
the duality pairing in (\ref{BH1-dual}), we have
$$
\bigl\langle \Phi, T_{(n)}\overline{\otimes} I_{S^1}\bigr\rangle \longrightarrow 
\bigl\langle\Phi,  T\overline{\otimes} I_{S^1} \bigr\rangle,
$$
for any $\Phi\in \frac{C_0(\R;S^\infty)}{Z_0\overset{\vee}{\otimes} S^\infty}\, 
\otimes\, H^1(\R;S^1)$.
Using (\ref{Unif}) again, we deduce that $T_{(n)}\overline{\otimes} I_{S^1}
\to T\overline{\otimes} I_{S^1}$ in the 
$w^*$-topology of $B(H^1(\R;S^1))$. Thus $T\in \overline{{\mathcal R}_{S^1}}^{w^*}$,
which proves (\ref{Haase}).
\end{proof}

Proposition \ref{Range} provides an isometric identification
\begin{equation}\label{AOS1-dual}
\A_{0,S^1}(\R)^*\,\simeq\, {\mathcal M}_{S^1}(H^1(\R))
\end{equation}
that will be used in Corollary \ref{QuotMap} below. A simple adaptation of the 
proof of Proposition \ref{Range} shows an analogous 
 isometric identification
\begin{equation}\label{AO-dual}
\A_{0}(\R)^*\,\simeq\, {\mathcal M}(H^1(\R)).
\end{equation}
More precisely, the first two authors showed in \cite[Theorem 3.7 and 
Remark 3.8]{AL}
that $\A_{0}(\R)^*\subset B(H^1(\R))$ is equal to the
$w^*$-closure of ${\mathcal R}$ in $B(H^1(\R))$, denoted by 
${\mathcal P\mathcal M}$. Then it follows from the proof of Proposition \ref{Range}
that ${\mathcal P\mathcal M}= {\mathcal \mathcal M}(H^1(\R))$.

\begin{rq1}\label{Diff} We noticed in Remark \ref{Not-cb} that
${\mathcal M}_{S^1}(H^1(\R))\not={\mathcal M} (H^1(\R))$. By the duality relations
(\ref{AO-dual}), (\ref{AOS1-dual}), this implies that the inequality
(\ref{A0A0S1}) cannot be reversed. More precisely, there is no
$\delta>0$ such that $\delta\norm{F}_{\A_{0}}\leq 
\norm{F}_{\A_{0,S^1}}$ for all $F\in\A_0(\R)$. Equivalently,
$$
\A_{0}(\R)\subsetneq \A_{0, S^1}(\R).
$$
\end{rq1}

The last part of this section is devoted to the realization of $\A_{0,S^1}(\R)$
as the range of a natural quotient map, completely different from $\Lambda$. 
Let 
$$
\mathcal{D}=\,\bigl\{\widehat{b}(-\,\cdotp)\, :\, b\in L^1(\R_+)\bigr\}.
$$
According to \cite[Proposition 3.9]{AL}, $\mathcal{D}\subset
\A_0(\R)$ and $\mathcal{D}$ is a dense subalgebra of $\A_0(\R)$.
The argument showing this density property shows as well
the following result.

\begin{lem}\label{Inclusions}
The algebra $\mathcal{D}$ is dense in $\A_{0,S^1}(\R)$. 
\end{lem}

The following is a consequence of \cite[Remark 3.10]{AL}.

\begin{lem}\label{Action}
Let $m$ be the symbol of an $S^1$-bounded Fourier multiplier
$T_m\colon H^1(\R)\to H^1(\R)$. Then for any $b\in L^1(\R_+)$,
$$
\bigl\langle T_m, \widehat{b}(-\,\cdotp)\bigr\rangle_{{\mathcal M}_{S^1}(H^1(\R)),
\A_{0,S^1}(\R)}\,=\,\int_{0}^{\infty} m(t)b(t)\, dt.
$$
\end{lem}

In the sequel, we use the fact that the convolution
$c\ast d$ of any two elements
$c,d\in L^1(\R_+)$ belongs to $L^1(\R_+)$.

\begin{cor}\label{QuotMap}
Let $\Theta\colon L^1(\R_+)\otimes L^1(\R_+)\to \mathcal{D}$ be the unique
linear map such that
$$
\Theta(c\otimes d) = \,\widehat{c\ast d}(-\,\cdotp),\qquad
c,d\in L^1(\R_+).
$$
Then $\Theta$ uniquely extends to a contractive map (still denoted by)
$$
\Theta\colon  L^1(\R_+)\otimes_{\gamma_2^*} L^1(\R_+)\longrightarrow
\A_{0,S^1}(\R).
$$
Moreover $\Theta$ is a quotient map (in the sense of Definition \ref{Quotient}).
\end{cor}

\begin{proof}
Let $\Psi\in L^1(\R_+)\otimes L^1(\R_+)$.
By the Hahn-Banach Theorem and (\ref{AOS1-dual}), there 
exists $T=T_m\in {\mathcal M}_{S^1}(H^1(\R))$ 
such that $\norm{T_m\overline{\otimes} I_{S^1}}=1$ and
$$
\norm{\Theta(\Psi)}_{\A_{0,S^1}} = 
\bigl\langle T_m, \Theta(\Psi)
\bigr\rangle_{{\mathcal M}_{S^1}(H^1(\R)),
\A_{0,S^1}(\R)}.
$$
Let us write $\Psi$ as
$$
\Psi = \sum_j c_j\otimes d_j
$$
for some finite families $(c_j)_j$ and $(d_j)_j$ in $L^1(\R_+)$, and
let 
$$
b=\sum_j c_j\ast d_j.
$$
Then by Lemma \ref{Action}, we have
$$
\norm{\Theta(\Psi)}_{\A_{0,S^1}} = \int_{0}^\infty m(t) b(t)\, dt.
$$
For any fixed $j$, we have
\begin{align*}
\int_{0}^\infty m(t) (c_j\ast d_j)(t) \, dt \,
& =\int_{0}^\infty m(t) \int_{0}^t c_j(s)d_j(t-s)\, ds\,dt\\
& =\int_{0}^\infty c_j(s)\int_{s}^\infty m(t) d_j(t-s)\, dt\,ds\\
& =\int_{0}^\infty \int_{0}^\infty m(s+u) c_j(s)d_j(u)\, du\,ds.
\end{align*}
Summing up over $j$, we deduce that
$$
\norm{\Theta(\Psi)}_{\A_{0,S^1}} = 
\,\sum_j \int_{0}^\infty \int_{0}^\infty m(s+u) c_j(s)d_j(u)\, du\,ds\,.
$$
Let $\overset{\circ}{m}\in L^\infty(\R_+^2)$ be defined
by $\overset{\circ}{m}(s,u) = m(s+u)$ and recall the definition of $ S_{\overset{\circ}{m}}$
given by (\ref{Kernel1}) and (\ref{Kernel2}).
Then we can rewrite the above equality as
$$
\norm{\Theta(\Psi)}_{\A_{0,S^1}} = 
\,\sum_j \bigl\langle S_{\overset{\circ}{m}}(d_j),c_j\bigr\rangle_{L^{\infty}(\R_+),
L^{1}(\R_+)}.
$$
By Theorem \ref{Main} and 
Lemma \ref{V2}, $S_{\overset{\circ}{m}}$ belongs to $\Gamma_2\bigl(L^{1}(\R_+),
L^{\infty}(\R_+)\bigr)$ with $\gamma_2(S_{\overset{\circ}{m}})\leq 1$. 
Hence
$$
\Bigl\vert \sum_j \bigl\langle S_{\overset{\circ}{m}}(d_j),
c_j\bigr\rangle_{L^{\infty}(\R_+),
L^{1}(\R_+)}\Bigr\vert\,\leq
\gamma_2^*\Bigl(\sum_j c_j\otimes d_j\Bigr).
$$
Hence we have established that
$$
\norm{\Theta(\Psi)}_{\A_{0,S^1}}\leq \gamma_2^*(\Psi).
$$
This proves that $\Theta$ extends
to a contraction $L^1(\R_+)\otimes_{\gamma_2^*} L^1(\R_+)\to
\A_{0,S^1}(\R)$. The uniqueness is clear.

Applying (\ref{AOS1-dual}) and (\ref{gamma2*}), we can regard the adjoint
of $\Theta$ as a map
$$
\Theta^*\colon 
{\mathcal M}_{S^1}(H^1(\R))\longrightarrow 
\Gamma_2\bigl(L^{1}(\R_+),
L^{\infty}(\R_+)\bigr).
$$
Furthermore, it is  easy to derive from the 
above calculation that
$\Theta^*(T_m) = S_{\overset{\circ}{m}}$
for any $T_m$ in ${\mathcal M}_{S^1}(H^1(\R))$.
By Theorem \ref{Main}, the map $\Theta^*$
is therefore an isometry. Consequently,
$\Theta$ is a quotient map.
\end{proof}

\section{Functional calculus estimates for bounded  $C_0$-semigroups on Hilbert space.}
\label{FC}

We assume that the reader is familiar with the basics of semigroup 
theory, for which we refer e.g. to \cite{Gol} or \cite{Paz}.

Throughout this section, we fix a Hilbert space $H$, 
a bounded $C_0$-semigroup $(T_t)_{t\geq 0}$
on $H$ and we let $-A$ be its infinitesimal generator.

For any $b\in L^1(\R_+)$, we let $L_b\colon\C_+\to\C$ be the Laplace transform 
of $b$ defined by
\begin{equation}\label{Lb}
L_b(z)=\,\int_{0}^{\infty} e^{-tz} b(t)\, dt,\qquad z\in\C_+.
\end{equation}
Following \cite{AL}, we define an operator
$\Gamma(A,b)\in B(H)$ by 
$$
[\Gamma(A,b)](x) = \int_0^\infty b(t) T_t(x)\, dt,\qquad x\in H.
$$
For simplicity, this operator will also be denoted by $\int_0^\infty b(t) T_t\, dt\,$ in the sequel.
The mapping $b\mapsto \Gamma(A,b)$ is the so-called Hille-Phillips functional
calculus (see e.g. \cite[Section 3.3]{Haase1}).
The operator $\Gamma(A,b)$ can be formally regarded as $L_b(A)$. The aim of
this section is to devise a sharp functional calculus on a
half-plane version 
$\A_{0,S^1}(\C_+)$ of $\A_{0,S^1}(\R)$
containing the $L_b$ and mapping 
$L_b$ to $\Gamma(A,b)$ for all $b\in L^1(\R_+)$.

We follow the pattern of \cite[Subsection 3.4]{AL} to define $\A_{0,S^1}(\C_+)$.
For any $F\in H^\infty(\R)$, we let
$$
\widetilde{F}\colon \C_+\longrightarrow\C
$$
be the bounded holomorphic function defined by 
setting
$$
\widetilde{F}(z) = \phi_F(iz),\qquad z\in\C_+,
$$
where $\phi_F\colon P_+\to\C$ is the Poisson integral
of $F$ (see also Subsection \ref{Hardy}).
We set
$$
\A_{0,S^1}(\C_+) = \bigl\{\widetilde{F}\, :\, F\in\A_{0,S^1}(\R)\bigr\}
$$
that we equip with the norm given by the rule
$$
\norm{\widetilde{F}}_{\A_{0,S^1}(\C_+)}
\,=\,
\norm{F}_{\A_{0,S^1}(\R)}.
$$
The mapping $F\mapsto \widetilde{F}$ is multiplicative
hence by
Proposition \ref{BanAlg}, $\A_{0,S^1}(\C_+)$ is 
a Banach algebra for pointwise multiplication.

If $F=\widehat{b}(-\,\cdotp)\colon\R\to\C$ for some 
$b\in L^1(\R_+)$, then $\widetilde{F}$ coincides with
$L_b$. As a consequence of Lemma \ref{Inclusions}, we
therefore obtain the following.

\begin{lem}\label{Inclusions-Tilde}
The space $\bigl\{L_b\, :\, b\in L^1(\R_+)\}$ is a 
dense subalgebra of $\A_{0,S^1}(\C_+)$.
\end{lem}

We set 
$$
C_A :=\,\sup\bigl\{\norm{T_t}\, :\, t\geq 0\bigr\}.
$$ 
The following is the main result of this section. 

\begin{thm}\label{FCA0S1}
There exists a unique bounded homomorphism $\rho_{0,S^1}^{A}\colon
\A_{0,S^1}(\C_+)\to B(H)$ such that
\begin{equation}\label{FCA0S1-2}
\rho_{0,S^1}^{A}(L_b)=\int_0^\infty b(t)T_t\,dt
\end{equation}
for all $b\in L^1(\R_+)$. Moreover $\norm{\rho_{0,S^1}^{A}}\leq C_A^2$.
\end{thm}

According to Remark \ref{Diff}, the above functional 
calculus $\rho_{0,S^1}^{A}\colon
\A_{0,S^1}(\C_+)\to B(H)$ is a non trivial extension
of the functional calculus $\A_{0}(\C_+)\to B(H)$ from
\cite[Theorem 4.4]{AL}.

The proof of Theorem \ref{FCA0S1} will rely on Corollary  
\ref{QuotMap}, hence on the description
of $S^1$-bounded multipliers of $H^1(\R)$ given in Section \ref{Multipliers}.
As far as we know, the first result connecting multipliers on Hardy spaces
with functional calculus estimates is due to Peller 
\cite{Pel}.

We will appeal to the half-plane holomorphic functional
calculus introduced by Batty-Haase-Mubeen in \cite{BHM}. 
Given any $\psi\in H^\infty(\C_+)$, and any $\epsilon>0$,
it allows the definition 
of a closed and densely defined operator $\psi(A+\epsilon)$
on $H$. We refer to \cite{BHM} 
for the construction and basic properties of this functional
calculus (see also \cite[Subsection 4.1]{AL}). We will use the fact, 
established in 
\cite[Lemma 4.2]{AL}, that for any $b\in L^1(\R_+)$, the operator
$L_b(A+\epsilon)$ provided by the half-plane holomorphic function 
calculus satisfies
\begin{equation}\label{Compatible}
L_b(A+\varepsilon) = \Gamma(A+\varepsilon,b).
\end{equation}

\begin{proof}[Proof of Theorem \ref{FCA0S1}]
We first prove that the linear mapping 
$L^1(\R_+)\otimes L^1(\R_+)\to B(H)$
taking $c\otimes d$ to $\Gamma(A,c\ast d)$ for any $c,d\in L^1(\R_+)$
extends to a bounded map
\begin{equation}\label{SigmaA}
\sigma_{A}\colon 
L^1(\R_+)\otimes_{\gamma_2^*} L^1(\R_+)\longrightarrow
B(H),\qquad\hbox{with}\quad \norm{\sigma_{A}}\leq C_A^2.
\end{equation}
This result is very close to 
some estimates in \cite[Section 5]{White}, to which we refer for
related results.

We consider an arbitrary 
$\Psi = \sum_j c_j\otimes d_j$ in $L^1(\R_+)\otimes L^1(\R_+)$
for some finite families $(c_j)_j$ and $(d_j)_j$ in $L^1(\R_+)$, and
we let 
$b=\sum_j c_j\ast d_j$. We have
$$
\sigma_A(\Psi) = \Gamma(A,b),
$$
by definition, hence to prove (\ref{SigmaA})  it
suffices to check that 
\begin{equation}\label{W}
\Bignorm{\int_{0}^{\infty} b(t)T_t\,dt}\,\leq C_A^2\,\gamma_2^*(\Psi).
\end{equation}
We regard $\Psi$ as an element of $L^1(\R_+^2)$. Then
\begin{equation}\label{b}
b(t) = \int_{0}^{t} \Psi(s,t-s)\, ds,\qquad t\in\R_+.
\end{equation}
Fixing $x,y\in H$, we derive from above that
\begin{align*}
\int_{0}^{\infty} b(t)\langle T_t(x),y\rangle_H\, dt\,
& =  \int_{0}^{\infty} \int_{0}^{t}\Psi(s,t-s) 
\langle T_t(x),y\rangle_H\,  ds\, dt\\
& =  \int_{0}^{\infty} \int_{s}^{\infty}\Psi(s,t-s) 
\langle T_t(x),y\rangle_H\,  dt\, ds\\
& =  \int_{0}^{\infty} \int_{0}^{\infty}\Psi(s,u) 
\langle T_{s+u}(x),y\rangle_H\,  du\, ds\\
& = \int_{0}^{\infty} \int_{0}^{\infty} \Psi(s,u) 
\langle T_{u}(x),T_{s}^*(y)\rangle_H\,  du\, ds.
\end{align*}
The function $u\mapsto T_{u}(x)$ is continuous, hence measurable,
from $\R_+$ into $H$ and we have $\norm{T_{u}(x)}\leq C_A\norm{x}$ for 
all $u\geq 0$. Likewise, $s\mapsto T_{s}^*(y)$ is a measurable function 
from $\R_+$ into $H$ satisfying $\norm{T_{s}^*(y)}\leq C_A\norm{y}$ for 
all $s\geq 0$. Hence the function 
$$
\varphi_{x,y}\colon (s,u)\,\mapsto \langle T_{u}(x),T_{s}^*(y)\rangle_H
$$
belongs to $\mathcal{V}_2(\R_+^2)$, with $\nu_2\bigl(\varphi_{x,y}\bigr)\leq
C_A^2\norm{x}\norm{y}$. According to Lemma \ref{gamma2*}, it therefore
follows from the above computation that 
$$
\Bigl\vert
\int_{0}^{\infty} b(t)\langle T_t(x),y\rangle_H\, dt\Bigr\vert
\,\leq\,C_A^2\norm{x}\norm{y}\,\gamma_2^*(\Psi).
$$
This proves (\ref{W}).

We will use an approximation
process appealing to the operators $A+\epsilon$, with
$\epsilon>0$. Note that $A+\epsilon$ is the negative 
generator of the $C_0$-semigroup $(e^{-\epsilon t}T_t)_{t\geq 0}$.
The latter is bounded hence we can apply the above reasoning
with $A+\epsilon$ instead of $A$. Thus for all $\epsilon>0$,
we have a bounded operator
$$
\sigma_{A+\epsilon}\colon 
L^1(\R_+)\otimes_{\gamma_2^*} L^1(\R_+)\longrightarrow
B(H)
$$
satisfying $\sigma_{A+\epsilon}(c\otimes d) = \Gamma(A+\epsilon,
c\ast d)$ 
for any $c,d\in L^1(\R_+)$. Observe that $C_{A+\epsilon}\leq C_A$. 
This yields a uniform estimate
\begin{equation}\label{Uniform}
\norm{\sigma_{A+\epsilon}} \leq C_A^2,\qquad \epsilon>0.
\end{equation}
It clearly follows from Lebesgue's dominated convergence
theorem that 
for any $\Psi$ in the tensor product
$L^1(\R_+)\otimes L^1(\R_+)$,
$\sigma_{A+\epsilon}(\Psi)\to \sigma_{A}(\Psi)$ strongly,
when $\epsilon\to 0$. Then (\ref{Uniform}) ensures that
for any 
$\Psi$ in the completed tensor product 
$L^1(\R_+)\otimes_{\gamma_2^*} L^1(\R_+)$,
we have
\begin{equation}\label{Approx}
\sigma_{A+\epsilon}(\Psi)\,\overset{\epsilon \to 0}{\longrightarrow}\,
\sigma_{A}(\Psi)\quad\hbox{strongly}.
\end{equation}

The mapping $\sigma_A$ is closely related to the quotient
map $\Theta$ given by Corollary \ref{QuotMap}. Indeed
for $\Psi \in L^1(\R_+)\otimes L^1(\R_+)$ and 
$b\in L^1(\R_+)$ given by (\ref{b}), we have 
$\Theta(\Psi) = \widehat{b}(-\,\cdotp)$ hence
$\widetilde{\Theta(\Psi)}= L_b$. 
Applying (\ref{Compatible}), this yields
$$
\widetilde{\Theta(\Psi)}(A+\epsilon) = 
\sigma_{A+\epsilon}(\Psi),\qquad \Psi\in 
L^1(\R_+)\otimes L^1(\R_+),\ \epsilon>0.
$$

We now claim that 
\begin{equation}\label{KK}
{\rm Ker}(\Theta)\subset {\rm Ker}(\sigma_A).
\end{equation}
To prove this, let 
$\Psi\in {\rm Ker}(\Theta)$ and let $(\Psi_n)_{n\geq 0}$
be a sequence in
$L^1(\R_+)\otimes L^1(\R_+)$
such that $\gamma_2^*\bigl(\Psi -\Psi_n\bigr)\to 0$
when $n\to\infty$. 
Then $\Theta(\Psi_n)\to \Theta(\Psi)=0$ in $\A_{0,S^1}(\R)$,
hence $\widetilde{\Theta(\Psi_n)}\to 0$
uniformly on $\C_+$, when $n\to\infty$. Further for all $\epsilon>0$, 
$\sigma_{A+\epsilon}(\Psi_n)\to \sigma_{A+\epsilon}(\Psi)$
in $B(H)$, when $n\to\infty$. It therefore follows from
\cite[Theorem 3.1]{BHM} (see also \cite[Lemma 4.1]{AL})
that $\sigma_{A+\epsilon}(\Psi)=0$, for all $\epsilon>0$.
Applying (\ref{Approx}), we deduce that 
$\sigma_{A}(\Psi)=0$, which proves (\ref{KK}).

We can now conclude the proof, as follows. 
Since $\Theta$ is a quotient map, by Corollary \ref{QuotMap}, it 
follows from the inclusion (\ref{KK}) that there 
exists a necessary unique bounded map
$\rho\colon \A_{0,S^1}(\R)\to B(H)$ such that
$$
\sigma_A = \rho\circ \Theta.
$$
Moreover, $\norm{\rho}=\norm{\sigma_A}\leq C_A^2$.
Let $\widetilde{\rho}\colon \A_{0,S^1}(\C_+)\to B(H)$
be derived from $\rho$ by writing
$\widetilde{\rho}(\widetilde{F})=\rho(F)$
for any $F\in \A_{0,S^1}(\R)$. 

Let $\Psi \in L^1(\R_+)\otimes L^1(\R_+)$ and 
let $b\in L^1(\R_+)$ be given by (\ref{b}). Then 
$$
\sigma_A(\Psi) =\rho\bigl(\Theta(\Psi)\bigr)
=\rho\bigl(\widehat{b}(-\,\cdotp)\bigr) =
\widetilde{\rho}(L_b),
$$
hence 
\begin{equation}\label{Ident}
\widetilde{\rho}(L_b)
= \int_0^\infty b(t)T_t\, dt.
\end{equation}
Any $b\in L^1(\R_+)$ can be approximated in $L^1$-norm
by functions of the form $b\ast c$, with $c\in L^1(\R_+)$.
Hence (\ref{Ident})
actually holds true for any $b\in L^1(\R_+)$.
Thus $\rho_{0,S^1}^{A} :=\widetilde{\rho}$ satisfies 
property (\ref{FCA0S1-2}) and the norm 
estimate in Theorem \ref{FCA0S1}. The fact that 
$\rho_{0,S^1}^{A}$ is a homomorphism and the uniqueness property 
are straightforward consequences of Lemma \ref{Inclusions}.
\end{proof}

\begin{rq1}\label{CB-Reg}
We observe that 
\begin{equation}\label{Differ}
{\mathcal M}_{S^1}(H^1(\R))\,\not=\,{\mathcal R}.
\end{equation}
Indeed let $\iota\colon \A_{0,S^1(\R)}\to C_0(\R)$ be the inclusion map
(this is a contraction, by (\ref{Contraction})). Let
$$
\iota^*\colon M(\R)\longrightarrow {\mathcal M}_{S^1}(H^1(\R))\,\subset B(H^1(\R;S^1))
$$
be its adjoint map,
with respect to the duality  pairings  (\ref{Riesz}) and  (\ref{BH1-dual}).
Then it follows from the identity (\ref{DualDual}) that 
$$
\iota^*(\mu) = R_\mu\overline{\otimes} I_{S^1},\qquad \mu\in M(\R).
$$
Assume that ${\mathcal R}$ and ${\mathcal M}_{S^1}(H^1(\R))$ are equal.
Then $\iota^*$
is an isomorphism, hence
$\iota$ itself is an isomorphism. Thus there exists a constant $\delta>0$
such that $\delta \norm{\,\cdotp}_{\A_{0,S^1}}\leq \norm{\,\cdotp}_\infty$
on $\A_{0,S^1}(\R)$. According to Theorem \ref{FCA0S1} this implies that
for any bounded $C_0$-semigroup $(T_t)_{t\geq 0}$ on Hilbert space, there exists
a constant $C>0$ such that
$$
\Bignorm{\int_0^\infty b(t)T_t\,dt}\,\leq\, C\norm{\widehat{b}}_\infty,\qquad
b\in L^1(\R_+).
$$
However the existence of bounded $C_0$-semigroups on Hilbert space
not satisfying this property is well-known. We refer to \cite[Subsection 4.4]{AL}
where this property is discussed at length. This yields a contradiction and shows
(\ref{Differ}).

We do not know any concrete example of an $S^1$-bounded Fourier
multiplier on $H^1(\R)$ not belonging to ${\mathcal R}$. Also we
do not know if ${\mathcal R}$ is dense in ${\mathcal M}_{S^1}(H^1(\R))$
for the  ${\mathcal M}_{S^1}(H^1(\R))$-norm.
\end{rq1}

We now mention an extension of Theorem \ref{FCA0S1} in the spirit of 
\cite[Subsection 4.3]{AL}. Let $BUC(\R; S^\infty)$ be the space of 
all bounded and uniformly continuous functions from $\R$
into $S^\infty$. Equipped with the supremum norm $\norm{\,\cdotp}_\infty$,
this is a Banach space. Furthermore, $C_0(\R; S^\infty)$
is a subspace of $BUC(\R; S^\infty)$. We let $\A_{S^1}(\R)$ be the space of all
functions $F \colon \R \rightarrow \C$ such that there 
exist two sequences $(f_k)_{k\geq 1}$ in 
$BUC(\R;S^\infty)$  
and $(h_k)_{k\geq 1}$  in $H^1(\R;S^1)$ satisfying (\ref{DefA0}). 
We set 
$$
\norme{F}_{\A_{S^1}} = \inf \Bigl\{\sum_{k=1}^{\infty} 
\normeinf{f_k}\norme{h_k}_1 \Bigr\},
$$
where the infimum runs over all sequences  $(f_k)_{k\geq 1}$ in 
$BUC(\R;S^\infty)$  
and $(h_k)_{k\geq 1}$  in $H^1(\R;S^1)$ satisfying (\ref{DefA0}). 
Then  $\A_{S^1}(\R)$ is a Banach space and by construction,
\begin{equation}\label{Embed}
\A_{0,S^1}(\R)\subset \A_{S^1}(\R).
\end{equation} 
The argument in \cite[Proposition 3.12]{AL} shows that
the embedding (\ref{Embed}) is actually an isometric one.

The proof of Proposition
\ref{BanAlg} shows as well that $\A_{S^1}(\R)$ is a Banach algebra.
Moreover for any $f_1\in BUC(\R; S^\infty)$ and $f_2\in C_0(\R; S^\infty)$, 
the function $\varphi_s\colon\R\to S_\infty\otimes_{\rm min} S^\infty$ 
defined by $\varphi_s(t)=f_1(t)\otimes f_2(t-s)$ belongs to 
$C_0(\R; S_\infty\otimes_{\rm min} S^\infty)$, for all $s\in\R$. Hence the proof of Proposition
\ref{BanAlg} actually shows that $\A_{0,S^1}(\R)$
is an ideal of $\A_{S^1}(\R)$ and that for all $F\in \A_{S^1}(\R)$ and 
$G\in\A_{0,S^1}(\R)$, we have
\begin{equation}\label{FG}
\norm{FG}_{\A_{0,S^1}}\leq 
\norm{F}_{\A_{S^1}}\norm{G}_{\A_{0,S^1}}.
\end{equation}

Turning to holomorphic functions, we set
$$
\A_{S^1}(\C_+) = \bigl\{\widetilde{F}\, :\, F\in\A_{S^1}(\R)\bigr\}
$$
that we equip with the norm given by 
$\norm{\widetilde{F}}_{\A_{0,S^1}(\C_+)}
=
\norm{F}_{\A_{0,S^1}(\R)}$. Thus 
$\A_{S^1}(\C_+)$ is 
a Banach algebra for pointwise multiplication, containing 
$\A_{0, S^1}(\C_+)$ as a closed ideal.

\begin{cor}\label{Extension}
There exists a unique bounded homomorphism
$\rho_{S^1}^{A}\colon \A_{S^1}(\C_+)\to B(H)$ 
such that
$$
\rho_{S^1}^{A}(L_b)=\int_0^\infty b(t)T_t\,dt,\qquad b\in L^1(\R_+).
$$
Moreover $\norm{\rho_{S^1}^{A}}\leq C_A^2$.
\end{cor}

\begin{proof}
We shall deduce this result from Theorem \ref{FCA0S1} with an argument similar to 
the one used to pass from \cite[Theorem 4.4]{AL} to  \cite[Corollary 4.6]{AL}. For any
integer $N\geq 1$, we let $G_N\colon\C_+\to\C$ be defined by
$G_N(z)=N(N+z)^{-1}$.  We noticed in the proof of \cite[Proposition 3.12]{AL} that
$(G_N)_{N\geq 1}$ is a contractive approximate unit of $\A_0(\C_+)$.
Since $\A_0(\C_+)$ is dense in $\A_{0,S^1}(\C_+)$, by Lemma \ref{Inclusions}, we deduce that
$(G_N)_{N\geq 1}$ is a contractive approximate unit of $\A_{0,S^1}(\C_+)$.

For any $\varphi\in \A_{S^1}(\C_+)$, the product $\varphi G_1$ belongs to $\A_{0,S^1}(\C_+)$
by the ideal property of $\A_{0, S^1}(\C_+)$. Hence
we may define an a priori unbounded operator 
$S_\varphi :=(1+A)\rho_{0,S^1}^{A}(\varphi G_1)$, with domain ${\rm Dom}(S_\varphi)
=\{x\in H\, :\, [\rho_{0,S^1}^{A}(\varphi G_1)](x)\in {\rm Dom}(A)\}$. 
Using the fact that $(G_N)_{N\geq 1}$ is a contractive approximate unit of $\A_{0,S^1}(\C_+)$
and the inequality (\ref{FG}), 
the argument in the proof of \cite[Corollary 4.6]{AL} shows that
${\rm Dom}(S_\varphi)=H$ and $S_\varphi\in B(H)$, with 
$$
\norm{S_\varphi}\leq C_A^2\norm{\varphi}_{\A_{S^1}}.
$$
Next one defines $\rho_{S^1}^{A}\colon \A_{S^1}(\C_+)\to B(H)$ by $\rho_{S^1}^{A}(\varphi)=S_\varphi$.
Arguing again as in the proof of \cite[Corollary 4.6]{AL}, we see that $\rho_{S^1}^{A}$ extends 
$\rho_{0,S^1}^{A}$ and that $\rho_{S^1}^{A}$ is multiplicative. This proves the
existence statement. 

To prove uniqueness, consider $\rho_{S^1}^{A}$ satisfying the properties of Corollary \ref{Extension}.
By Lemma \ref{Inclusions-Tilde}, $\rho_{S^1}^{A}$ extends $\rho_{0,S^1}^{A}$. Hence
for any $\varphi\in \A_{S^1}(\C_+)$,  we have
$$
\rho_{0,S^1}^{A}(\varphi G_1)= 
\rho_{S^1}^{A}(\varphi G_1)=\rho_{0,S^1}^{A} (G_1)\rho_{S^1}^{A}(\varphi)
= (1+A)^{-1}\rho_{S^1}^{A}(\varphi).
$$ 
This implies that 
$\rho_{S^1}^{A}(\varphi)=(1+A)\rho_{0,S^1}^{A}(\varphi G_1)=S_\varphi$.
\end{proof}

We finally note that the proof of Theorem \ref{FCA0S1} provides a new  proof of 
\cite[Theorem 4.4]{AL}, different from the original one. On the
other hand, it is possible to write a proof of Theorem \ref{FCA0S1} 
following the approach of \cite{AL}. However this is a lengthy proof
and there is no point in writing it here.

\vskip 1cm
\noindent
{\bf Acknowledgements.} LA was supported by the ERC
grant {\it Rigidity of groups and higher index theory} under the European Union’s
Horizon 2020 research and innovation program (grant agreement no. 677120-INDEX).
CL was supported by the ANR project {\it Noncommutative analysis on groups and quantum groups}
(No./ANR-19-CE40-0002). The  authors
gratefully thank the Heilbronn Institute for Mathematical Research (HIMR) 
and the UKRI/EPSRC Additional Funding Programme for Mathematical Sciences 
for the financial support through a Heilbronn Small Grant award.

\
\vskip 0.8cm

\end{document}